\documentclass[11pt]{article}
\usepackage{amssymb,amsmath,amscd}
\usepackage{amsthm}
\newtheorem{theorem}{Theorem}[section]
\newtheorem{theorem*}{Theorem A\!\!}
\newtheorem{proposition}{Proposition}[section]
\newtheorem{corollary}{Corollary}[section]
\newtheorem{definition}{Definition}[section]
\newtheorem{proposition*}{Proposition A\!\!}
\newtheorem{corollary*}{Corollary A\!\!}
\newtheorem{lemma}{Lemma}[section]

\DeclareMathOperator{\Hom}{Hom}

\DeclareMathOperator{\Id}{Id}

\begin{document}

\title {Singular conformally invariant trilinear forms, I\\ the multiplicity one theorem}

\author{Jean-Louis Clerc}

\date{February 15, 2015}

 \maketitle

\begin{abstract}{A \emph{normalized holomorphic family} (depending on $\boldsymbol \lambda \in \mathbb C^3$) of conformally invariant trilinear forms on the sphere is studied. Its zero set $Z$ is described. For $\boldsymbol \lambda \notin Z$, the multiplicity of the space of conformally invariant trilinear forms is shown to be 1.}
\end{abstract}
{2010 Mathematics Subject Classification : 22E45, 43A80}\\
Key words : conformal covariance, trilinear form, meromorphic family of distributions, transverse differential operator.

\section*{Introduction}

The present paper is a continuation of the study of conformally invariant trilinear forms for three representations belonging to the scalar principal series of the conformal group $G=SO_0(1,n)$. More precisely, the representations are realized on $\mathcal C^\infty(S)$, where $S\simeq S^{n-1}$ is the unit sphere in a $n$-dimensional Euclidean space $E$, and  trilinear forms are required to be continuous for the natural topology on $\mathcal C^\infty(S)\times \mathcal C^\infty(S)\times \mathcal C^\infty(S)$.

In the reference \cite{co}, co-authored with B. \O rsted, the \emph{generic} case was studied. The trilinear forms $\mathcal K^{\boldsymbol \lambda}$ are formally defined by an integral formula, depending on a parameter $\boldsymbol \lambda = (\lambda_1,\lambda_2,\lambda_3)$ in $\mathbb C^3$, each $\lambda_j$ indexing a representation of the scalar principal series of $G$. The domain of convergence is described, and the meromorphic extension is obtained, showing simple poles along four families of parallel equally spaced planes in $\mathbb C^3$. A \emph{generic multiplicity one theorem} (valid for $\boldsymbol \lambda$ not a pole) is proved for the space of trilinear conformally invariant forms.

In a second article \cite{bc}, coauthored with R. Beckmann, the residues along the planes of poles were  computed, at least generically. The planes of poles are partitioned in two types, those of \emph{type }I and those of \emph{type }II. Viewed as a distribution on $S\times S\times S$, the residue of $\mathcal K^{\boldsymbol \lambda}$ at a pole of type I is a singular distribution, supported on a submanifold of codimension $(n-1)$, its expression uses \emph{covariant differential operators} on $S$, while the residue at a pole of type II is supported on a submanifold of codimension $2(n-1)$ (the diagonal in $S\times S\times S$) and its expression requires \emph{covariant bi-differential operators} (mapping functions  on $S\times S$ to functions on $S$).

In the present paper the \emph{normalized } holomorphic family $\widetilde {\mathcal K}^{\boldsymbol \lambda}$ (obtained from $\mathcal K^{\boldsymbol \lambda}$ by multiplication by appropriate inverse $\Gamma$ factors) is introduced and studied. The first result is the determination of the zero set $Z$ of the holomorphic function $\boldsymbol \lambda \mapsto \widetilde {\mathcal K}^{\boldsymbol \lambda}$. The set $Z$ is contained in the set of poles and  is of codimension 2 ($Z$ is a denumerable union of lines in $\mathbb C^3$). The determination of $Z$ uses $K$-harmonic analysis on $S\times S\times S$, where $K\simeq SO(n)$ is the maximal compact subgroup of $G$. The explicit computation of some integrals (called Bernstein-Reznikov integrals) (see \cite{br,d,co, ckop}) plays a major role in this approach.

The second main result is an extension of the generic multiplicity one theorem. For $\boldsymbol \lambda\in \mathbb C^3$, let $Tri(\boldsymbol \lambda)$ be the space of conformally invariant trilinear forms with respect to the three representations indexed by $\boldsymbol \lambda$. 
\begin{theorem}\label{mult1}
\[\boldsymbol \lambda \in \mathbb C^3 \setminus Z \Longrightarrow \dim Tri(\boldsymbol \lambda) = 1 \ .\]
\end{theorem}
An equivalent formulation is that $Tri(\boldsymbol \lambda) =\mathbb C\widetilde K^{\boldsymbol \lambda}$ for $\boldsymbol \lambda\notin Z$. 

To obtain this result, we first reformulate Bruhat's theory (\cite{b}). The concept of a distribution \emph{smoothly supported on a submanifold} is introduced, which allows, for such a distribution, the definition of its \emph{transverse symbol} (see \cite{kv} for similar ideas). This gives more flexibility for studying "nearly invariant" distributions supported on a submanifold, and this is crucial in section 6. The strategy for the proof of Theorem \ref{mult1} is presented with more details at the beginning of section 6.

T. Oshima (personal communication) observed that  $\dim Tri(\boldsymbol \lambda) \geq 2$ for $\boldsymbol \lambda\in Z$. This (Theorem \ref{mult2})is obtained by using a general result on the closure of a meromorphic family of distributions (Lemma 6.3 in \cite{o}). Hence 
\[\boldsymbol \lambda \in \mathbb C^3 \setminus Z \Longleftrightarrow \dim Tri(\boldsymbol \lambda) = 1 \ .
\]

Throughout this paper, we assume that $n\geq 4$. The methods of the present paper could be (to the price of slight modifications) developed also for $n=2$ or $n=3$, but these two cases are special. This can be observed already in  Liouville's theorem on local conformal diffeomorphisms, for which $S^1$ and $S^2$ stand apart. Perhaps a deeper insight into the difference can be obtained from the geometric structure of spheres when viewed as \emph{symmetric R-spaces}. Recall that a symmetric R-space is a homogeneous space $S=G/P$ where $G$ is a semi-simple Lie group, and  $P$ a parabolic subgroup, such that, denoting by $K$ a maximal compact subgroup of $G$, the space $S$, viewed as $S\simeq K/(K\cap P)$ is a compact Riemannian \emph{symmetric} space. As a major result,  a symmetric R-space is a \emph{real form of  a compact Hermitian space}. Now observe that $S^1$ is a real form of $P^1(\mathbb C)$, which is a Hermitian compact symmetric space of rank 1, $S^2\simeq P^1(\mathbb C)$  \emph{is} a compact Hermitian symmetric space (hence a real form of $P^1(\mathbb C) \times \overline{P}^1(\mathbb C)$, a product of two Hermitian symmetric spaces of rank 1), whereas, for $d\geq 3$, $S^d$ is a real form of the complex projective quadric $Q^d(\mathbb C)$, which is an irreducible compact Hermitian symmetric space of rank 2 (see \cite{c} for more details on the geometry and analysis of symmetric R-spaces). For $n=3$, which corresponds to the Riemann sphere $S^2\simeq P^1(\mathbb C)$ as a homogenous space for $SL_2(\mathbb C)$, the conformally invariant trilinear forms where studied by Oksak in \cite {o1} (even more generally for representations of the principal series which are not necessarily scalar). He obtained fairly complete results, making heavy use of the complex nature of $S^2$, and his techniques do not seem to be transposable to  higher dimensional spheres.

Let us mention connections to two other problems. First, 
the space $Tri(\boldsymbol \lambda)$ is isomorphic to the space $\Hom_G(\pi_{\lambda_1}\otimes \pi_{\lambda_2}, \pi_{-\lambda_3})$ (consequence of a lemma due to Poulsen, see \cite {p}). From this point of view, the present study is a special case of the \emph{restriction program} designed by T. Kobayashi and collaborators, where the "big group" is $G\times G$ and the "small group" is $diag(G)\simeq G$, the representations are $\pi_{\lambda_1}\otimes \pi_{\lambda_2}$ as a representation of $G\times G$, and $\pi_{-\lambda_3}$ as a representation of $G$. In this respect, the reference \cite{ks} was a source of inspiration for the present paper. A second problem concerns those invariant trilinear forms, which, when viewed as distributions on $S\times S\times S$, are supported in the diagonal.  They are expressed by \emph{covariant bi-differential operators}. Those have been studied (see \cite{or,s}), also in connection with problems about tensor product of generalized Verma modules (see \cite{s, koss},) and they deserve further study. 
\medskip

\centerline{\bf Summary}
\medskip

{\bf 0. Introduction}
\smallskip

{\bf 1.} {\bf Conformal geometry and analysis on the sphere}

\hskip 1cm {1.1} Geometry of the sphere

\hskip 1cm {1.2} The scalar principal series  

\hskip 1cm {1.3} Geometry of orbits in $S\times S\times S$

{\bf 2.} {\bf The holomorphic family $\widetilde {\mathcal K}^{\boldsymbol \lambda}$}

\hskip 1 cm {2.1} The construction of the generic family

\hskip 1 cm {2.2} The Bernstein-Reznikov integrals

{\bf 3.} {\bf $K$-analysis of $\widetilde {\mathcal K}^{\boldsymbol \lambda}$}

\hskip 1 cm 3.1 $K$-invariant polynomial functions on $S\times S\times S$

\hskip 1 cm 3.2 $K$-analysis of $\widetilde {\mathcal K}^{\boldsymbol \lambda}$

\hskip 1 cm 3.3 The support of $\widetilde {\mathcal K}^{\boldsymbol \lambda}$

{\bf 4.} {\bf The zero set $Z$ of $\widetilde {\mathcal K}^{\boldsymbol \lambda}$}

\hskip 1 cm  4.1 The first vanishing situation for $\widetilde {\mathcal K}^{\boldsymbol \lambda}$

\hskip 1 cm 4.2 The second vanishing situation for $\widetilde {\mathcal K}^{\boldsymbol \lambda}$

\hskip 1 cm 4.3 Necessary conditions for the vanishing of  $\widetilde {\mathcal K}^{\boldsymbol \lambda}$

{\bf 5.} {\bf Distributions smoothly supported on a submanifold}
\medskip

{\bf 6.} {\bf The theorem of the support}

\hskip 1 cm 6.1 Type I

\hskip 1 cm 6.2 Type II

\hskip 1 cm 6.3 Type I+II

{\bf 7.} {\bf Multiplicity 1 results}

\hskip 1cm 7.1 Type I

\hskip 1cm 7.2 Type II and type I+II

\hskip 1cm 7.3 Multiplicity one theorem for bi-differential operators

\hskip 1cm 7.4 Final remarks

{\bf References}

\section{Conformal geometry and analysis on the sphere}

\subsection{ Geometry of the sphere}
Let $S\simeq S^{n-1}$ be the unit sphere in a Euclidean space $E$ of dimension $n$. The Euclidean distance is denoted by $\vert x-y\vert$. Let $G=SO_0(n,1)$ be the connected component of the neutral element in the Lorentz group. The action of $G$ on $S$ is by conformal transformations. For $g \in G$ and $x\in S$, the \emph{conformal factor} $\kappa(g,x)$ of $g$ at $x$ is defined by the relation
\begin{equation} \vert Dg(x)\xi \vert = \kappa(g,x)\,\vert\xi\vert\ ,
\end{equation}
for any vector $\xi$ in the tangent space $T_x(S)$ of $S$ at $x$.

The restriction to the sphere of the Euclidean distance satisfies an important \emph{covariance relation} under the action of $G$, namely

\begin{equation}\label{covdist}
\vert g(x)-g(y)\vert = \kappa(g,x)^{\frac{1}{2}}\, \vert x-y\vert\, \kappa(g,y)^{\frac{1}{2}}
\end{equation}
for $x,y\in S, g\in G$.

The subgroup $K\simeq SO(n)$  can be identified with a maximal compact subgroup of $G$. It acts transitively on $S$. Let $dx$ be the measure on $S$ induced by the Euclidean structure. It transforms under the action of $G$ by the rule
$d(g(x)) = \kappa(g,x)^{n-1} dx$, for $g\in G$. 

Let $E^{1,n} =\mathbb R\oplus E$ equipped with the Lorentzian form
$ [(x_0, x)] = x_0^2-\vert x\vert^2$. The Lie algebra $\mathfrak g\simeq \mathfrak o(1,n)$ of the Lie group $G$ is realized in matrix form as

\[\mathfrak = \Bigg\{\begin{pmatrix} 0& &u^t& \\ & & &\\ u& &X&\\ & & & \end{pmatrix},\quad u\in E, X\in \mathfrak o(E)\Bigg\}\ .
\]
The Cartan decomposition of $\mathfrak g$ w.r.t. the standard Cartan involution is $\mathfrak g = \mathfrak k\oplus \mathfrak p$, where
\[\mathfrak k = \Bigg\{\begin{pmatrix} 0&0 \\ 0&X \end{pmatrix}, X\in \mathfrak o(E)\Bigg\}\simeq \mathfrak o(n), \quad \mathfrak p = \Bigg\{X_u=\begin{pmatrix} 0&u^t \\ u&0\end{pmatrix},\quad u\in E\Bigg\}\simeq E
\]
Fix an orthonormal basis $\{e_1,e_2,\dots, e_n\}$ of $E$, and choose $\mathbf 1 = (1,0,\dots, 0)$ as the origin in $S$. The stabilizer of $\mathbf 1$ in $G$ is a parabolic subgroup $P$ which contains in particular the Abelian subgroup  $(a_t)_{t\in \mathbb R}$, where
\[ a_t = \begin{pmatrix} \cosh t& \sinh t&\dots&0&\dots\\\sinh t& \cosh t&\dots&0&\dots\\0&0& 1&&\\ \vdots&\vdots& &\ddots&\\0&0&\dots& &1\end{pmatrix}\ .
\]
Notice that $a_t =\exp( tX_{e_1})$. These elements act on $S$ by the formula
\[x= \begin{pmatrix} x_1\\x_2\\\vdots\\x_n\end{pmatrix}
 \longmapsto a_t(x) = \begin{pmatrix}\frac{ \sinh t + x_1 \cosh t}{\cosh t +x_1\sinh t}\\ \frac{ x_2}{\cosh t +x_1\sinh t}\\\vdots\\\frac{ x_n}{\cosh t +x_1\sinh t} \end{pmatrix} \ ,\]
 and $\kappa(a_t, x) = (\cosh t+x_1 \sinh t)^{-1}$.
To any $X\in \mathfrak p$, we denote by $\widetilde X$ the vector field on $S$ given by
\[\widetilde X f(x) = \frac{d}{dt} f(\exp tX(x))_{\vert t=0}\ .
\] 
\begin{lemma} Let $X\in \mathfrak p\simeq E$. Then
\begin{equation}\label{derkappa}
\frac{d}{dt} \kappa(\exp tX, x)_{\vert t=0} =-\langle X,x\rangle\ .
\end{equation}
\end{lemma}
\begin{proof}
Let $X=X_{e_1}$. Then $\exp tX = a_t$, and $\kappa(a_t,x) = (\cosh t+x_1 \sinh t)^{-1}=1-x_1t+O(t^2)$. Hence $\displaystyle \frac{d}{dt} \kappa(\exp tX, x)_{\vert t=0} = -x_1= -\langle X,x\rangle$.
The general case follows, as it is always possible to choose the orthonormal basis of $E$ such that $X$ is (a multiple of) $e_1$.
\end{proof}
\subsection{The scalar principal series}

For $\lambda\in \mathbb C$, let $\pi_\lambda$ be the representation (\emph{spherical principal series}) realized on $\mathcal C^\infty(S)$ by
\[\pi_\lambda(g) f(x) = \kappa(g^{-1}, x)^{\rho+\lambda} f\big(g^{-1}(x)\big)\ ,
\]
where $\rho=\frac{n-1}{2}$. Observe that for $k\in K$, $\pi_\lambda(k)$ is independant of $\lambda$ and coincides with the regular action of $K$ on $\mathcal C^\infty(S)$.

\begin{lemma} Let $X\in \mathfrak p$. Then 
\begin{equation}\label{dpi}
d\pi_{\lambda}(X) f(x) = -\widetilde X f (x)+(\lambda +\rho) \langle X,x\rangle f(x)\ .
\end{equation}
\end{lemma}
\begin{proof}
\[ d\pi_\lambda(X) f (x) = \frac{d}{dt}\big( \kappa(\exp -tX),x)^{\lambda+\rho} f(\exp-tX(x))\big)_{\vert t=0}\]
\[ = -\widetilde X f(x) +(\lambda+\rho) \langle X,x\rangle f(x)
\]
by using \eqref{derkappa}.
\end{proof}
Recall the \emph{duality relation} for $\pi_\lambda$ and $\pi_{-\lambda}$ given by
\begin{equation}\label{duality}
 \int_S \big(\pi_{-\lambda}(g) \varphi\big)(x) \,\psi(x)\, dx = \int_S \varphi(x) \big(\pi_\lambda(g^{-1})\,\psi\big)(x) \,dx\ ,
\end{equation}
where $\varphi,\psi\in \mathcal C^\infty(S), g\in G$.

Finally, recall that the representation $\pi_\lambda$ is irreducible, unless \[\lambda \in \big(-\rho-\mathbb N\big) \cup \big( \rho+\mathbb N\big)\ .\] 

 \subsection{ Geometry of orbits in $S\times S\times S$}
The action of $G$ can be extended "diagonally" to $S\times S\times S$ by $g(x,y,z) = \big(g(x),g(y),g(z)\big)$. The next proposition recalls the structure of $G$-orbits for this action.

\begin{proposition} There are five orbits in $S\times S\times S$ under the action of $G$. Denoting by $(x,y,z)$ a generic element of $S\times S\times S$, they are given by

\[\mathcal O_0 = \{ x\neq y, y\neq z, z\neq x\}\] 
\[\mathcal O_1 = \{  y=z\neq x\}, \quad \mathcal O_2 = \{  z=x\neq y\}, \quad \mathcal O_3 = \{  x=y\neq z\}\ ,
\]
\[\mathcal O_4 = \{ x=y=z\}\ .\]
The orbit $\mathcal O_0$ is open and dense in $S\times S\times S$,  $\mathcal O_4$  is closed, and for $j=1,2,3$, $\overline {\mathcal O_j} = \mathcal O_j \cup \mathcal O_4$. \end{proposition}

See e.g. \cite {co} Proposition 1.2.\footnote{When $n=2$, there are \emph{six} orbits, \emph{two} of them being open.}

\section{The holomorphic family $\widetilde {\mathcal K}^{\boldsymbol \lambda}$}

Let $\lambda_1,\lambda_2,\lambda_3 \in \mathbb C$. A continuous trilinear form $\mathcal T$ on $\mathcal C^\infty(S)\times \mathcal C^\infty(S)\times \mathcal C^\infty(S)$ is \emph{invariant w.r.t. $(\pi_{\lambda_1}, \pi_{\lambda_2}, \pi_{\lambda_3})$} if
\[\mathcal T\big(\pi_{\lambda_1}(g) f_1, \pi_{\lambda_2}(g) f_2,\pi_{\lambda_3}(g) f_3\big)= \mathcal T(f_1,f_2,f_3)
\]
for all $f_1,f_2,f_3\in \mathcal C^\infty(S)$ and $g\in G$. 

By Schwartz's kernel theorem, such a trilinear form can also be viewed as a distribution on $S\times S\times S$, still denoted by $\mathcal T$. Let $\boldsymbol \lambda = (\lambda_1,\lambda_2,\lambda_3)$ and let ${\boldsymbol \pi}_{\boldsymbol \lambda} $  be the representation on $\mathcal C^\infty(S\times S\times S)$ (= the canonical completion of the tensor product $\pi_{\lambda_1}\otimes \pi_{\lambda_2}\otimes \pi_{\lambda_3}$) given by
\[{\boldsymbol \pi}_{\boldsymbol \lambda} (g) f (x,y,z) = \kappa(g^{-1},x)^{\lambda_1+\rho}\kappa(g^{-1},y)^{\lambda_2+\rho}\kappa(g^{-1},z)^{\lambda_3+\rho} f( g^{-1}(x), g^{-1}(y), g^{-1}(z))\ .
\]
Then the invariance condition simply reads
\[\mathcal T({\boldsymbol \pi}_{\boldsymbol \lambda}(g) f) = \mathcal T(f)\]
for all $f\in \mathcal C^\infty(S\times S \times S)$ and $g\in G$.
When this is verified, we also say simply that $\mathcal T$  is \emph{$\boldsymbol \lambda$-invariant}.

For $\boldsymbol \lambda\in \mathbb C^3$, let $Tri(\boldsymbol \lambda)$ be the space of $\boldsymbol \lambda$-invariant distributions on $S\times S\times S$.

\subsection{The construction of the generic family}

In \cite{co}, B. \O rsted and the present author constructed trilinear invariant forms for the  family of representations $(\pi_\lambda)_{\lambda\in \mathbb C}$. Given $\boldsymbol {\lambda} =(\lambda_1,\lambda_2,\lambda_3) \in \mathbb C^3$, let $\boldsymbol {\alpha} = (\alpha_1,\alpha_2,\alpha_3)$ be given by 
\begin{equation}\label{alphalambda}
\begin{split}
\alpha_1 = -\rho-\lambda_1+\lambda_2+\lambda_3\\
\alpha_2 = -\rho +\lambda_1-\lambda_2+\lambda_3\\
\alpha_3 = -\rho+\lambda_1+\lambda_2-\lambda_3
\end{split}
\end{equation}

Relations \eqref{alphalambda} can be inverted to yield
\begin{equation}\label{lambdaalpha}
\begin{split}
\lambda_1 = \rho+ \frac{\alpha_2+\alpha_3}{2}\\
\lambda_2 = \rho+ \frac{\alpha_3+\alpha_1}{2}\\
\lambda_3 = \rho+ \frac{\alpha_1+\alpha_2}{2}\ .
\end{split}
\end{equation}

In the sequel, $\boldsymbol \alpha = (\alpha_1,\alpha_2,\alpha_3)$ (called the \emph{geometric parameter} describing the singularities of $\mathcal K_{\boldsymbol \alpha}$) and $\boldsymbol \lambda = (\lambda_1,\lambda_2,\lambda_3)$ (called the \emph{spectral parameter} describing the invariance of  $\mathcal K^{\boldsymbol \lambda}$) will be considered as two associated parameters on the same space $\mathbb C^3$, related by the relations  \eqref{alphalambda} or \eqref{lambdaalpha}.

Let $\boldsymbol \lambda\in \mathbb C^3$ and let $\boldsymbol \alpha$ be its associated geometric parameter. For $f\in \mathcal C^\infty(S\times S\times S)$, let
\begin{equation}
\label{K}\mathcal K^{\boldsymbol \lambda}(f)=
\mathcal K_{\boldsymbol \alpha} (f) = \iiint_{S\times S\times S}\hskip-0.4cm \vert x-y\vert^{\alpha_3} \vert y-z\vert^{\alpha_1}\vert z-x\vert^{\alpha_2} f(x,y,z)\, dx \,dy \, dz\, .
\end{equation}
The formula defines (formally) a map on $\mathcal C^\infty(S\times \mathcal S\times S)$, which is \emph{invariant} under $\boldsymbol \pi_{\boldsymbol \lambda}$. Less formally, assume that $Supp(f)\subset \mathcal O_0$. Then the integral makes sense and thus defines a distribution $\mathcal K_{\boldsymbol \alpha,\, \mathcal O_0}$ on the open orbit $\mathcal O_0$.

The integral  \eqref{K} is absolutely convergent for all $f \in \mathcal C^\infty(S\times S\times S)$ if and only if
\[\Re \alpha_j > -(n-1),\ j=1,2,3,\qquad \Re (\alpha_1+\alpha_2+\alpha_3) >-2(n-1)\ .\]
The integral can be meromorphically continued to $\mathbb C^3$, with simple poles along four families of planes in $\mathbb C^3$, 

\[  \alpha_j = -(n-1)-2 k_j,\quad k_j \in \mathbb N\, \]
for $j=1,2$ or $3$ 
and 
\[ \alpha_1+\alpha_2+\alpha_3 = -2(n-1) -2k,\quad k\in \mathbb N\ .\]

A pole $\boldsymbol \alpha$ is said to be 
\smallskip

$\bullet$ \emph{of type } I${ }_j$ if $ \alpha_j \in -(n-1)-2\mathbb N$ (and in general \emph{of type }I if it is of type I$_j$ for some $j\in \{1,2,3\}$)
\smallskip

$\bullet$ \emph{of type } II if $\alpha_1+\alpha_1+\alpha_3 \in -2(n-1)-2\mathbb N$
\smallskip

$\bullet$ \emph{of type } I+II if $\boldsymbol \alpha$ is at the same time a pole of type I and a pole of type II.

\begin{definition}\label{generic}
A pole $\boldsymbol \alpha$ is said to be \emph{generic} if $\boldsymbol \alpha$ belongs to a \emph{unique} plane of poles.  
\end{definition}
Introduce the \emph{normalized} trilinear invariant functional $\widetilde {\mathcal K}_{\boldsymbol \alpha}$ defined by

\[\widetilde {\mathcal K}_{\boldsymbol \alpha}= \frac{\mathcal K_{\boldsymbol \alpha}}{\Gamma(\rho+\frac{\alpha_1}{2})\Gamma(\rho+\frac{\alpha_2}{2})\Gamma(\rho+\frac{\alpha_3}{2})\Gamma(\frac{\alpha_1+\alpha_2+\alpha_3}{2}+2\rho)}\ \ ,
\]
or similarly,
\[\widetilde K^{\boldsymbol \lambda} = \frac{K^{\boldsymbol \lambda}}{\Gamma(\frac{\lambda_1+\lambda_2+\lambda_3+\rho}{2})\Gamma(\frac{-\lambda_1+\lambda_2+\lambda_3+\rho}{2})\Gamma(\frac{\lambda_1-\lambda_2+\lambda_3+\rho}{2})\Gamma(\frac{\lambda_1+\lambda_2-\lambda_3+\rho}{2})}\ \ .
\]
Both definitions define \emph{entire} (distribution-valued) functions in $\mathbb C^3$, as the generic poles of $\mathcal K_{\boldsymbol \alpha}$ are simple and by using Hartog's prolongation principle.

Notice that a pole $\boldsymbol \alpha$ is generic if and only if among the four $\Gamma$ factors in the normalization process, exactly one is singular.
\subsection {The Bernstein-Reznikov integrals}

The following calculation was achieved in \cite{co} (see also \cite{ckop, d}), extending an earlier result in the case $n=2$ (see \cite{br}).
\begin{proposition}\label{BRint} 

Let $\boldsymbol \alpha = (\alpha_1,\alpha_2,\alpha_3)$ and assume it is not a pole. Then

\begin{equation*}
\begin{split}
&\int_{S\times S \times S} \vert x-y\vert^{\alpha_3} \vert y-z\vert^{\alpha_1} \vert z-x\vert^{\alpha_2} \,dx\,dy\,dz \ =\\
&\big(\frac{\pi}{2}\big)^{\frac{3}{2}(n-1)} 2^{\alpha_1+\alpha_2+\alpha_3}\,\frac{\Gamma\big(\frac{\alpha_1+\alpha_2+\alpha_3}{2} +2\rho\big)\,\Gamma(\frac{\alpha_1}{2}+\rho)\,\Gamma(\frac{\alpha_2}{2}+\rho)\,\Gamma(\frac{\alpha_3}{2}+\rho)}{\Gamma(\frac{\alpha_1+\alpha_2}{2}+2\rho)\,\Gamma(\frac{\alpha_2+\alpha_3}{2}+2\rho)\,\Gamma(\frac{\alpha_3+\alpha_1}{2}+2\rho)}
\end{split}
\end{equation*}

\end{proposition}
This formula can be equivalently written as
\begin{equation}\label{BRalpha}
\widetilde K_{\boldsymbol \alpha} (1,1,1) = \frac{\big(\frac{\pi}{2}\big)^{\frac{3}{2}(n-1)} 2^{\alpha_1+\alpha_2+\alpha_3}}{\Gamma(\frac{\alpha_1+\alpha_2}{2}+2\rho)\,\Gamma(\frac{\alpha_2+\alpha_3}{2}+2\rho)\,\Gamma(\frac{\alpha_3+\alpha_1}{2}+2\rho)}
\end{equation}
or, in terms of the spectral parameter
$\boldsymbol \lambda = (\lambda_1,\lambda_2,\lambda_3)$
 \begin{equation}\label{ival2}
 \widetilde {\mathcal K}^{\boldsymbol \lambda}(1,1,1)= \left({\frac{\sqrt {\pi}}{2}}\right)^{3(n-1)}
\frac{ 2^{\lambda_1+\lambda_2+\lambda_3}}
 {\Gamma(\rho+\lambda_1)\,\Gamma(\rho+\lambda_2)\,\Gamma(\rho+\lambda_3)} \ .
\end{equation}

\section{$K$-analysis of $\widetilde {\mathcal K}^{\boldsymbol \lambda}$}

\subsection{ $K$-invariant polynomial functions on $S\times S\times S$}
The group $K\simeq SO(n)$ acts on $S$ by rotations. A \emph{$K$-invariant} function on $S\times S\times S$ is  a function on $S\times S\times S$ which is invariant under the diagonal action of $K$ on $S\times S\times S$. A $K$-invariant distribution on $S\times S\times S$ is defined by duality.

A \emph{polynomial function} function on $S$ (resp. $S\times S\times S$) is by definition the restriction to $S$ (resp. to $S\times S\times S$) of a polynomial on $E$(resp. $E\times E\times E$). 

\begin{lemma}\label{Kinv}
 A $K$-invariant polynomial function on $S\times S\times S$ is the restriction to $S\times S\times S$ of a $K$-invariant polynomial on $E\times E\times E$.
\end{lemma}

\begin{proof} Let $P$ be a polynomial on $E\times E\times E$, and assume that its restriction to $S\times S\times S$, say $p=P_{\vert S\times S\times S}$ is $K$-invariant. Then let for $x,y,z\in E$
\[Q(x,y,z) = \int_{K} P(kx,ky,kz)\, dk\ ,
\]
where $dk$ is the normalized Haar measure on $K$. Then $Q$ is a $K$-invariant polynomial. When $x,y,z$ belong to $S$, $Q(x,y,z) = \int_{K} p(kx,ky,kz) \, dk = p(x,y,z)$ by the $K$-invariance of $p$, hence $p=Q_{\vert S\times S\times S}$.
\end{proof}
With some abuse, we will often use the same notation for a polynomial on $E\times E\times E$ and its restriction to $S\times S\times S$.

\begin{lemma}\label{Kpoldense}
 The space of $K$-invariant polynomial functions is dense in the space of $K$-invariant functions in $\mathcal C^\infty(S\times S\times S)$.

\end{lemma}
\begin{proof} By $K\times K\times K$-Fourier analysis, the space of polynomial functions on $S\times S\times S$ is dense in $\mathcal C^\infty(S\times S\times S)$. As the topology on $\mathcal C^\infty(S\times S\times S)$ can be defined by $K$-invariant semi-norms (e.g. the Sobolev semi norms $\Vert (\Delta_x+\Delta_y+\Delta_z)^N f \Vert_2$), $K$-invariant polynomial functions are dense in the space of $K$-invariant functions in $\mathcal C^\infty(S\times S\times S)$.

\end{proof}

\begin{lemma}\label{Kgen}
 The algebra of $K$-invariant polynomial functions on $S\times S\times S$ is generated (as an algebra) by the polynomial functions \[\quad \vert x-y\vert^2,\quad \vert y-z\vert^2,\quad \vert z-x\vert^2\ .\]
\end{lemma}
\begin{proof} By the \emph{first fundamental theorem} (see e.g. \cite{gw}), the algebra of $K$-invariant polynomials on $E\times E\times E$ is generated by the polynomials
\[ \vert x\vert^2,\quad \vert y\vert^2,\quad \vert z\vert^2,\quad <x,y>, \quad<y,z>,\quad <z,x>\ .
\]
Hence, by Lemma \ref{Kinv}, the algebra of $K$-invariant polynomial functions on $S\times S\times S$ is generated by the restrictions to $S\times S\times S$ of the previous polynomials, that is to say by
\[1,\quad <x,y>, \quad<y,z>,\quad <z,x>\ .
\]
But for $x,y\in S$, $<x,y> = 1-\frac{1}{2}\vert x-y\vert^2$, so that  $\{\vert x-y\vert^2, \vert y-z\vert^2, \vert z-x\vert^2\}$ is also a generating family.
\end{proof}

\noindent
{\bf Remark. }The assumption $n\geq 4$ is important for this lemma. When $n=3$,  the algebra of polynomials on $E\times E\times E$ which are invariant under $SO(3)$  is generated by $\vert x\vert^2, \vert y\vert^2, \vert z\vert^2, <x,y>, <y,z>, <z,x>$ and $\det(x,y,z)$. For $n\geq 4$, the algebra of $SO(n)$-invariant polynomials is the same as the algebra of $O(n)$-invariant polynomials.
\medskip

\begin{lemma}\label{Kdist0}
 Let $\mathcal T$ be a $K$-invariant distribution on $S\times S\times S$. Then $\mathcal T \equiv 0$ if and only if $(\mathcal T,p)=0$ for any $K$-invariant  polynomial function $p$ on $S\times S\times S$.
\end{lemma}

\begin{proof} The only if part being trivial, assume that $\mathcal T$ satisfies $(\mathcal T,p)=0$ for any $K$-invariant polynomial function $p$ on $S\times S\times S$. Lemma \ref{Kpoldense} implies that $(\mathcal T,\varphi)=0$ for all $K$-invariant $\mathcal C^\infty$ functions $\varphi$ on $S\times S\times S$. Let $f\in \mathcal C^\infty (S\times S\times S)$. For $k\in K$, let $f_k$ be the function defined by
$f_k(x,y,z)) = f(kx,ky,kz)$. By $K$-invariance of $\mathcal T$, $(\mathcal T,f_k) = (\mathcal T,f)$. Let  $\varphi$ be  defined by $\varphi(x,y,z)) = \int_K f(k x,ky,kz))dk$. Then $\varphi$ is a $K$-invariant $\mathcal C^\infty$ function on $(S\times S\times S)$. Now
 \[(\mathcal T,f) = (\mathcal T, \int_K f_k\, dk) = (\mathcal T,\varphi) =0\ ,\]
hence $\mathcal T=0$.
 \end{proof}
\subsection{$K$-analysis of $\widetilde {\mathcal K}^{\boldsymbol \lambda}$}

For $a_1,a_2,a_3\in \mathbb N$, let $p_{a_1,a_2,a_3}$ be the polynomial function on $S\times S\times S$ defined by
\[p_{a_1,\,a_2,\,a_3} (x,y,z) = \vert x-y\vert^{2a_3} \vert y-z\vert^{2a_1} \vert z-x\vert^{2a_2}\ .
\]
\begin{proposition}\label{Kvan}
 Let $\mathcal T$ be a $K$-invariant distribution on $S\times S\times S$. Then $\mathcal T\equiv 0$ if and only if $\mathcal T (p_{a_1,\,a_2\,,a_3}) =0$ for any $a_1,a_2,a_3\in \mathbb N$.
\end{proposition}

\begin{proof}
The proposition is a consequence of   Lemma \ref{Kgen} and Lemma \ref{Kdist0}.
\end{proof}

Recall the \emph{Pochhammer symbol} $(x)_n$ which is defined for $x\in \mathbb C$ and $n\in \mathbb N$ by 
\[(x)_0 = 1,\quad (x)_1 = x,\quad (x)_n = x(x+1)\dots (x+n-1)\ . \]
Observe that $(x)_n=0$ if and only if $x\in  -\mathbb N$ and $n> -x$.

\begin{proposition} Let $\boldsymbol \alpha\in \mathbb C^3$, and let $a_1,a_2,a_3\in \mathbb N$.

\begin{equation}\label{Kpalpha}
I_{\boldsymbol \alpha} (a_1,a_2,a_3) : = \widetilde {\mathcal K}_{\boldsymbol \alpha} (p_{a_1,\,a_2,\,a_3}) =\quad \big(\frac{\pi}{2}\big)^{\frac{3}{2}(n-1)} 2^{\alpha_1+\alpha_2+\alpha_3}2^{ 2(a_1+a_2+a_3)}\dots
\end{equation}
\begin{equation*}
\frac{ \big(\frac{\alpha_1+\alpha_2+\alpha_3}{2}+2\rho\big)_{a_1+a_2+a_3}
\big(\frac{\alpha_1}{2} +\rho\big)_{a_1}\big(\frac{\alpha_2}{2} +\rho\big)_{a_2}\big(\frac{\alpha_3}{2} +\rho\big)_{a_3}} 
{\Gamma(\frac{\alpha_1+\alpha_2}{2}+2\rho+a_1+a_2)\,\Gamma(\frac{\alpha_2+\alpha_3}{2}+2\rho+a_2+a_3)\,\Gamma(\frac{\alpha_3+\alpha_1}{2}+2\rho+a_3+a_1)}\ .
\end{equation*}
\end{proposition}

\begin{proof} Suppose first that $\boldsymbol \alpha$ is not a pole of $\mathcal K_{\boldsymbol \alpha}$. Then \[\mathcal K_{\boldsymbol \alpha}(p_{a_1,a_2,a_3}) = \mathcal K_{\alpha_1+2a_1,\,\alpha_2+2a_2, \,\alpha_3 +2a_3}(1\otimes1\otimes1)\ ,\]
whose value is known by \eqref{BRalpha}. Then take into account the normalizing factor, to get \eqref{Kvan} for $\boldsymbol \alpha$ not a pole. The two handsides of \eqref{Kpalpha} being holomorphic on $\mathbb C^3$ are equal for all $\boldsymbol \alpha$ in $\mathbb C^3$.
\end{proof}

For further use, we also give the same result formulated in terms of the spectral parameter.
\begin{equation}\label{Kplambda}
\begin{split}
&I^{\boldsymbol \lambda}(a_1,a_2,a_3) = \widetilde {\mathcal K}^{\boldsymbol \lambda} (p_{a_1,a_2,a_3}) \\ &=(\frac{\sqrt{\pi}}{2})^{3(n-1)}2^{\lambda_1+\lambda_2+\lambda_3}2^{2(a_1+a_2+a_3)}\big(\frac{\lambda_1+\lambda_2 +\lambda_3+\rho}{2}\big)_{a_1+a_2+a_3}\\
&\frac{(\frac{-\lambda_1+\lambda_2 +\lambda_3+\rho}{2})_{a_1}(\frac{\lambda_1-\lambda_2 +\lambda_3+\rho}{2})_{a_2}(\frac{\lambda_1+\lambda_2 -\lambda_3+\rho}{2})_{a_3}}
{\Gamma(\lambda_1 +\rho+a_2+a_3)\Gamma(\lambda_2 +\rho+a_3+a_1)\Gamma(\lambda_3 +\rho+a_1+a_2)}\ .
\end{split}
\end{equation}

\subsection{The support of $\widetilde {\mathcal K}^{\boldsymbol \lambda}$}

\begin{proposition}\label{suppK} Let $\boldsymbol \lambda\in\mathbb C^3$.
\smallskip

$i)$ if $\boldsymbol \lambda$ is not a pole, $Supp(\widetilde K^{\boldsymbol \lambda} )=S\times S\times S$.
\smallskip

$ii)$ if $\boldsymbol \lambda$ is a pole of type I$_j, j\in \{1,2,3\}$, then
$Supp( \widetilde K^{\boldsymbol \lambda} )\subset \overline {\mathcal O_j}$. 
\smallskip

$iii)$ if  $\boldsymbol \lambda$ is a \emph{generic} pole of type I$_j, j\in \{1,2,3\}$, then $Supp( \widetilde K^{\boldsymbol \lambda} )= \overline {\mathcal O_j}$.
\smallskip

$iv)$ if $\boldsymbol \lambda$ is a pole of type II,
then $Supp(\widetilde K^{\boldsymbol \lambda} )\subset \mathcal O_4$.
\end{proposition}

\begin{proof} The support of any $\boldsymbol \lambda$-invariant distribution is a closed $G$-invariant subset of $S\times S\times S$. If $\boldsymbol \lambda$ is not a of pole, then $\widetilde K^{\boldsymbol \lambda} $ is a non zero multiple of $K^{\boldsymbol \lambda}$ and hence has the same support. The restriction of  $K^{\boldsymbol \lambda}$ to $\mathcal O_0$ is certainly not $0$, and hence $Supp(K^{\boldsymbol \lambda}) = S\times S\times S$. Hence $i)$ is verified.

Let $\boldsymbol \lambda_0$ be a generic pole (recall Definition \ref{generic}), either of type $I$ or of type $II$. The normalizing factor in the definition of $\widetilde K^{\boldsymbol \lambda}$ is a product of three  $\Gamma$ factors which are non singular at $\boldsymbol \lambda_0$ and only one $\Gamma$ factor which has a simple pole at $\boldsymbol \lambda_0$.  Hence, $\widetilde K^{\boldsymbol \lambda_0}$ is equal (up to a non zero scalar) to the residue of $K^{\boldsymbol \lambda}$ at $\boldsymbol \lambda_0$. The residues are computed in \cite{bc}, at least generically. More precisely, in each plane of poles, the residues are computed for $\boldsymbol \lambda$ in a dense open subset. From the formul\ae   \ for the residues, it is easy to deduce the inclusion of their support as indicated in the proposition. The general result follows by analytic continuation  in each plane of poles. Hence  $ii)$ and $iv)$ hold true.

Assume now that  $\boldsymbol \lambda$ is a \emph{generic} pole of type I, e.g. of type $I_3$. As $Supp(\widetilde {\mathcal K}^{\boldsymbol \lambda})\subset  \overline {\mathcal O_3}$ and $Supp(\widetilde K^{\boldsymbol \lambda})$ is invariant under $G$, either $Supp(\widetilde {\mathcal K}^{\boldsymbol \lambda})= \overline {\mathcal O_3}$ or $Supp(\widetilde {\mathcal K}^{\boldsymbol \lambda}) \subset \mathcal O_4$. So, assume that $Supp(\widetilde {\mathcal K}^{\boldsymbol \lambda}) \subset\mathcal O_4$. By Schwartz's theorem on the local structure of distributions supported by a submanifold and  the compactness of $S$, if a test function $\varphi\in \mathcal C^\infty(S\times S\times S)$ vanishes on $\mathcal O_4$ at a sufficiently large order, then $\widetilde {\mathcal K}^{\boldsymbol \lambda}(\varphi)=0$. For $a_1,a_2\in \mathbb N$  large enough, the function $p_{a_1,\,a_2,\,0}$ vanishes on $\mathcal O_4$ at an arbitrary large order. A careful  inspection of $\eqref{Kvan}$ shows that, for $a_1,a_2$ large, $\widetilde {\mathcal K}^{\boldsymbol \lambda}(p_{a_1,\,a_2,\,0})\neq 0$, thus getting a contradiction. Hence $Supp(\widetilde {\mathcal K}^{\boldsymbol \lambda})=  \overline {\mathcal O_3}$.
\end{proof}

\section{ The zero set $ Z$ of $\widetilde {\mathcal K}^{\boldsymbol \lambda}$}

\begin{proposition} For $\boldsymbol \alpha \in \mathbb C^3$, the trilinear form $\widetilde {\mathcal K}_{\boldsymbol \alpha}$ is identically $0$ if and only if $I_{\boldsymbol \alpha} (a_1,a_2,a_3) = 0$ for any $a_1,a_2,a_3\in \mathbb N$.
\end{proposition}
\begin{proof} Viewed as a distribution on $S\times S\times S$, $\widetilde {\mathcal K}_{\boldsymbol \alpha}$ is $K$-invariant. Hence we may apply Proposition \ref{Kvan}.
\end{proof}

The rest of this section is devoted to the (rather long) proof of the following theorem.

\begin{theorem} \label{Z} Let $\boldsymbol \alpha = (\alpha_1,\alpha_2,\alpha_3)\in \mathbb C^3$. The trilinear form $\widetilde {\mathcal K}_{\boldsymbol \alpha}$ vanishes identically if and only if (at least) one of the following properties (up to a permutation of the indices) is satisfied

\begin{equation}\label{ZIalpha}
\alpha_1 =-(n-1)-2 k_1,\qquad \alpha_2 = -(n-1) -2k_2
\end{equation}
 for some $k_1,k_2\in \mathbb N$, or
\begin{equation}\label{ZIIalpha}
\alpha_1+\alpha_2+\alpha_3 = -2(n-1) -2k, \quad \alpha_3 = 2l_3
\end{equation}  
for some $k, l_3\in \mathbb N$.
\end{theorem}

This theorem can be translated  in terms of the spectral parameter $\boldsymbol \lambda$.

\begin{theorem}\label{ZK} Let $\boldsymbol \lambda = (\lambda_1,\lambda_2,\lambda_3)$. Then $\widetilde {\mathcal K}^{\boldsymbol \lambda}$ vanishes identically if and only if (at least) one of the following properties (up to permutation of the indices) is satisfied
\begin{equation}\label{ZIlambda}
\lambda_3 = -\rho-l, \quad \lambda_1-\lambda_2 = m,  \quad l\in \mathbb N, m\in \mathbb Z,\quad  \vert m\vert \leq l,\quad l\equiv m\  (2)
\end{equation}
\begin{equation}\label{ZIIlambda}
\lambda_3 = -\rho-l,\quad \lambda_1+\lambda_2 = m,\quad l\in \mathbb N,  m\in \mathbb Z,\quad \vert m\vert\leq l,\quad l\equiv m\ (2)
\end{equation}

\end{theorem}

\noindent
\emph{Proof} of the equivalence  of the two formulations.

$\bullet$  First case. Let $\alpha_1=-(n-1)-2k_1$ and $\alpha_2 = -(n-1)-2k_2$ where $k_1,k_2\in \mathbb N$. Then, by \eqref{lambdaalpha}
\[ \lambda_1 = -k_2 +\frac{\alpha_3}{2},\quad\lambda_2 = -k_1+\frac{\alpha_3}{2},\quad \lambda_3 = -\rho-k_1-k_2\ .
\]
Set $l=k_1+k_2$, and $m=k_1-k_2$, so that $\lambda_3 = -\rho-l, \lambda_1-\lambda_2 = m$. Then $l\in \mathbb N$, $m\in \mathbb Z, l\equiv m \ (2)$, and $\vert m\vert \leq l$.

Conversely, if $\boldsymbol \lambda$ satisfies \eqref{ZIlambda}, set $k_1 = \frac{l-m}{2}, k_2 = \frac{l+m}{2}$. Then $k_1,k_2\in \mathbb N$ and $\alpha_1 = -(n-1)-2 k_1, \alpha_2 = -(n-1) -2k_2$.
\medskip

$\bullet$ Second case. Let $\alpha_1+\alpha_2+\alpha_3 = -2(n-1)-2k$ and $\alpha_3=2l_3$, for some $k,l_3\in \mathbb N$. Then, by \eqref{lambdaalpha}
\[\lambda_1 = \rho+l_3+\frac{\alpha_2}{2},\quad \lambda_2 = -\rho-k-\frac{\alpha_2}{2}, \quad \lambda_3 = -\rho-k-l_3\ .
\]
Set $l=k+l_3, m=l_3-k$. Then $l\in \mathbb N, m\in \mathbb Z, m\equiv l\ (2)$ and $ \vert m\vert\leq l$. Moreover,
$\lambda_1+\lambda_2 = l_3-k=m, \lambda_3 = -\rho-l$. 

Conversely, if $\boldsymbol \lambda$ satisfies \eqref{ZIIlambda}, set $k= \frac{l-m}{2}, l_3=\frac{l+m}{2}$. Then $k, l_3\in \mathbb N $ and the statement follows.

\subsection{The first vanishing situation}

\begin{proposition}\label{van1}
 Let $\boldsymbol \alpha$ be such that \[\alpha_1 = -(n-1)-2k_1, \quad \alpha_2=-(n-1)-2k_2, \qquad  k_1,k_2\in \mathbb N\ .\]  Then
$\widetilde {\mathcal K}_{\boldsymbol \alpha} \equiv 0$.

\end{proposition}

\begin{proof} The assumptions on $\boldsymbol \alpha$ imply that
\[\frac{\alpha_1}{2} +\rho = -k_1,\quad \frac{\alpha_2}{2} +\rho = -k_2,\quad
\frac{\alpha_1+\alpha_2}{2} +2\rho = -k_1-k_2\] 
so that, for $a_1,a_2,a_3\in \mathbb N$,  \eqref{Kpalpha}  can be rewritten as
\[I_{\boldsymbol \alpha}(a_1,a_2,a_3) =
\frac{(-k_1)_{a_1} (-k_2)_{a_2}}{\Gamma(-k_1-k_2+a_1+a_2)}\quad \times \dots
\]
If $a_1+a_2\leq k_1+k_2$, the denominator is singular, so $I_{\boldsymbol \alpha}(a_1,a_2,a_3) =0$. If $a_1+a_2 > k_1+k_2$, then  either $a_1>k_1$ or $a_2>k_2$. If for instance $a_1>k_1$, then $(-k_1)_{a_1} = 0$, so that $I_{\boldsymbol \alpha}(a_1,a_2,a_3) =0$. \end{proof}

\subsection {The second vanishing situation}
\begin{proposition}\label{van2}
 Let $\boldsymbol \alpha$ be such that \[\alpha_1+\alpha_2+\alpha_3 = -2(n-1)-2k, \quad \alpha_3 = 2l_3,\quad  k,l_3\in \mathbb N\ .\] Then $\widetilde {\mathcal K}_{\boldsymbol \alpha} \equiv 0$.
\end{proposition}

\begin{proof} The assumptions on $\boldsymbol \alpha$ imply that

\[\frac{\alpha_1+\alpha_2+\alpha_3}{2} +(n-1) = -k,\qquad \frac{\alpha_1+\alpha_2}{2} +(n-1) = -k-l_3\]
so that for $a_1,a_2,a_3\in \mathbb N$, \eqref{Kpalpha} can be rewritten as 
\[I_{\boldsymbol \alpha}(a_1,a_2,a_3) = \frac{ (-k)_{a_1+a_2+a_3}}{\Gamma(-k-l_3+a_1+a_2)}\quad \times \dots
\]
If $a_1+a_2\leq k+l_3$,  $(-k-l_3+a_1+a_2)$ is a pole of $\Gamma$, so that $I_{\boldsymbol \alpha}(a_1,a_2,a_3) =0$. If $a_1+a_2>k+l_3$,  then $a_1+a_2+a_3> k+l_3\geq k$, so that $(-k)_{a_1+a_2+a_3}=0$.  Hence, in any case, $I_{\boldsymbol \alpha}(a_1,a_2,a_3) =0$.
\end{proof}

\subsection{ Necessary conditions for the vanishing}

We now prove that the conditions of Theorem \ref{Z} are \emph{necessary} to have $\widetilde {\mathcal K}_{\boldsymbol \alpha} \equiv0$. 

{\bf Step 1}. Reduction of the problem

\begin{proposition} Let $\boldsymbol \alpha$ be such that $\widetilde {\mathcal K}_{\boldsymbol \alpha} \equiv 0$. Then either $\boldsymbol \alpha$ belongs to $Z$ or, up to a permutation of the indices,
\[ \alpha_1+\alpha_2+\alpha_3 =-2(n-1)-2k, \quad \alpha_1+\alpha_2 = -2(n-1)-2l\]
for some $k,l\in \mathbb N, k>l$. 
\end{proposition}
\begin{proof} If $\boldsymbol \alpha$ is not a pole, then $\widetilde {\mathcal K}_{\boldsymbol \alpha}$ is a (non zero) multiple of ${\mathcal K}_{\boldsymbol \alpha}$, and the restriction of ${\mathcal K}_{\boldsymbol \alpha}$ to the open orbit $\mathcal O_0$ is non identically $0$. Hence $\widetilde {\mathcal K}_{\boldsymbol \alpha} \equiv 0$ implies that $\boldsymbol \alpha$ is a pole.

Next,  $\widetilde {\mathcal K}_{\boldsymbol \alpha} \equiv 0$ implies $I_{\boldsymbol \alpha}(0,0,0)=0$. Hence some $\Gamma$ factor in \eqref{BRalpha} has a pole, or explicitly
  \begin{equation}\label{Kvan2}
  \alpha_i+\alpha_j +2(n-1)\in -2\mathbb N
  \end{equation} 
  for some $1\leq i\neq j\leq3$.
\smallskip

Hence for $\boldsymbol \alpha$ a pole of type II, up to a permutation of the indices, there exist $k,l\in \mathbb N$, such that
\[ \alpha_1+\alpha_2+\alpha_3 =-2(n-1)-2k, \quad \alpha_1+\alpha_2 = -2(n-1)-2l\ .\]

Let $\boldsymbol \alpha$ be a pole of type $I$, say of type $I_1$, i.e. $\alpha_1 = -(n-1)-2k_1$, for some $k_1\in \mathbb N$. For $a_2$ and $a_3$ sufficiently large, \eqref{Kpalpha} implies
\[I_{\boldsymbol \alpha}(0,a_2,a_3) = \big(\frac{\alpha_1+\alpha_2+\alpha_3}{2}+2\rho\big)_{a_1+a_2+a_3}
\big(\frac{\alpha_2}{2} +\rho\big)_{a_2}\big(\frac{\alpha_3}{2} +\rho\big)_{a_3}\times NVT
\]
(NVT = non vanishing term).
 Hence $\widetilde {\mathcal K}_{\boldsymbol \alpha} \equiv 0$ implies
 \smallskip
 
 $\bullet$ $\frac{\alpha_2}{2} +\rho\in -\mathbb N$ (or $\frac{\alpha_3}{2} +\rho\in -\mathbb N$) 
 
 or
 \smallskip
  
 $\bullet$ $\frac{\alpha_1+\alpha_2+\alpha_3}{2}+2\rho \in -\mathbb N$.
\smallskip

In the first case, $\boldsymbol \alpha$ is also a pole of type $I_2$ (or $I_3$),  hence $\boldsymbol \alpha$ belongs to $Z$ by \eqref{ZIalpha}. In the second case, we are back to the case of a pole of type II.

Summing up, and up to a permutation of the indices, it remains to consider the situation where 
\[\alpha_1+\alpha_2+\alpha_3 = -2(n-1)-2k, \quad \alpha_1+\alpha_2 = -2(n-1)-2l\ .
\]
for some $k,l\in \mathbb N$. The two conditions imply $\alpha_3 = -2k+2l$. If $k\leq l$, then $\boldsymbol \alpha$ is in $Z$. So, only the case where $k>l$ remains open, and this is the content of the proposition.
\end{proof}

\begin{proposition} Let $\boldsymbol \alpha$ satisfy \[ \alpha_1+\alpha_2+\alpha_3 =-2(n-1)-2k, \quad \alpha_1+\alpha_2 =-2(n-1)-2l\]
for some $k,l\in \mathbb N, k>l $.
Then either $\boldsymbol \alpha$ belongs to $Z$ or there exists $a_1,a_2,a_3$ such that $\widetilde {\mathcal K}_{\boldsymbol \alpha}(p_{a_1,a_2,a_3})\neq 0$.
\end{proposition}
\begin{proof}
By \eqref{Kpalpha} $I_{\boldsymbol \alpha}(a_1,a_2,a_3)$ can be written as
\[\frac{(-k)_{a_1+a_2+a_3}\ (\frac{\alpha_1}{2}+\rho)_{a_1}\  (\frac{\alpha_2}{2}+\rho)_{a_2}(\frac{\alpha_3}{2}+\rho)_{a_3}}{\Gamma(-l+a_1+a_2)\ \Gamma(\frac{\alpha_2+\alpha_3}{2} +(n-1) +a_2+a_3)\ \Gamma(\frac{\alpha_3+\alpha_1}{2}+(n-1) +a_3+a_1)}
\]
Choose $a_3=0$ and $a_1, a_2$ such that $l< a_1+a_2\leq k$. Then $(-k)_{a_1+a_2+a_3}\neq 0$ and $\Gamma(-l+a_1+a_2)$ is finite, so that $I_{\boldsymbol \alpha}(a_1,a_2,a_3)$  is $\neq 0$ unless (perhaps) one of  the following is true :
\medskip

$\bullet$ $\alpha_1 \in  -(n-1)-2\mathbb N$   or $\alpha_2 \in -(n-1) -2\mathbb N$ 

$\bullet$  $ \alpha_2+\alpha_3 \in -2(n-1) -2\mathbb N$ or $\alpha_3+\alpha_1 \in -(n-1)-2\mathbb N$.
\medskip

Up to a permutation of the indices $1$ and $2$, we are reduced to examine the following cases:
\smallskip

$\bullet$ case (A) \[\alpha_1+\alpha_2+\alpha_3= -2(n-1)-2k, \alpha_1+\alpha_2 = -2(n-1)-2l, \alpha_1 =-(n-1)-2m\]
for some $k,l,m\in \mathbb N, k>l$.
\smallskip

$\bullet$ case (B) \[\alpha_1+\alpha_2+\alpha_3= -2(n-1)-2k, \alpha_1+\alpha_2 =-2(n-1)-2l,\alpha_1+\alpha_3 = -2(n-1)-2m\]
for some $k,l,m\in \mathbb N, k>l, k>m$.

In each case, we have to prove that either $\boldsymbol \alpha$ is in $Z$ or there exists $a_1,a_2,a_3$ such that $I_{\boldsymbol \alpha}(a_1,a_2,a_3)\neq 0$.
\medskip

{\bf Step 2.} Case (A)

From the assumptions, we get
\[\alpha_2 = -(n-1)-2(l-m)\ .\]
If $l\geq m$, then $\boldsymbol \alpha$ is in two planes of poles of type I, hence is in $Z$. So assume 
 that $l<m$. Conditions $(A)$ imply
\[\alpha_1+\alpha_3 = -(n-1) -2k-2l-2m, \quad \alpha_2+\alpha_3 = -(n-1)+2(m-k)\ .
\]

$\bullet$ Suppose that $n-1$ is odd. Then $\alpha_3=2l-2k\notin -(n-1) -2\mathbb N$. So $I_{\boldsymbol \alpha}(a_1,a_2,a_3)$ is equal to

\[\frac{ (-k)_{a_1+a_2+a_3}\, (-m)_{a_1}}{\Gamma(-l+a_1+a_2)} \quad \times NVT\] 
(NVT = non vanishing term).
As $l<k$ and $l<m$, it is possible to choose $a_1\in \mathbb N$ such that $a_1>l$ and $a_1\leq k, a_1\leq m$. Let $a_2=a_3=0$. Then $(-k)_{a_1+a_2+a_3} \neq 0, (-m)_{a_1}\neq 0$, and $ \Gamma(-l+a_1+a_2)$ is finite, thus $I_{\boldsymbol \alpha}(a_1,a_2,a_3)\neq0$. 
\smallskip

$\bullet$ Suppose that $n-1$ is even. Now $\alpha_3 = -(n-1) +2(\frac{n-1}{2} -k+l)$. Let   $p = \frac{n-1}{2} -k+l\in \mathbb Z$. If $p\leq 0$, then $\boldsymbol \alpha$ lies in two planes of poles of type I and hence belongs to $Z$. Assume that $p>0$. Now,
\[ \frac{\alpha_1+\alpha_3}{2} +n-1= \frac{n-1}{2} -m -k+l = -m+p\]
\[ \frac{\alpha_2+\alpha_3}{2} +(n-1) = \frac{n-1}{2}+m-k\ .\]
But $\frac{n-1}{2}+m-k >\frac{n-1}{2}+l-k>0$. Hence $I_{\boldsymbol \alpha}(a_1,a_2,a_3)$ is equal to
\[\frac{(-k)_{a_1+a_2+a_3} (-m)_{a_1}}{\Gamma(-l+a_1+a_2) \Gamma(-m+ p +a_1+a_3)}\ \times NVT
\]
First choose $a_1$ such that $m-p<a_1\leq m$ and $l\leq a_1$. Now choose $a_2$ such that $l<a_1+a_2\leq k$, and let $a_3=0$. It follows that $I_{\boldsymbol \alpha}(a_1,a_2,a_3) \neq 0$.
\medskip

{\bf Step 3}. {Case (B)}

We now assume that $\boldsymbol \alpha$ satisfy the conditions
\[\alpha_1+\alpha_2+\alpha_3 = -2(n-1)-2k,\quad \alpha_1 +\alpha_2 =-2(n-1)  -2l, \quad \alpha_1+\alpha_3 =-2(n-1) - 2m \ ,
\]
 where $l<k$ and $m<k$.
 
 We may further assume that $\alpha_j\notin -(n-1)-2\mathbb N, j=1,2,3$, otherwise we are in case $A$.
 
Now $\alpha_1 = -2(n-1)+2k-2m-2l$. If $k\geq (n-1)+l+m$, then $\alpha_1\in 2\mathbb N$, so that  $\boldsymbol \alpha\in Z$ by \eqref{ZIIalpha}. So assume that $k<(n-1)+l+m$. Then
$\alpha_2+\alpha_3 = -2(2k-m-l)$.

$\bullet$ Assume first that $2k-m-l<(n-1)$. Then 
\[ I_{\boldsymbol \alpha}(a_1,a_2,a_3) = \frac{(-k)_{a_1+a_2+a_3}}{ \Gamma(-l+a_1+a_2)\Gamma(-m+a_1+a_3)} \times NVT\ .
\]
For $a_1= \sup(l,m)+1, a_2=a_3=0$, the $\Gamma$ factors are finite and the factor in the numerator is not $0$ because $l,m<k$.
\smallskip

$\bullet$ Assume now that $2k-m-l\geq (n-1)$, so that altogether
\[k<l+m+(n-1)\leq 2k\ .\]
Then
\[ I_{\boldsymbol \alpha}(a_1,a_2,a_3)\]\[ = \frac{(-k)_{a_1+a_2+a_3}}{ \Gamma(-l+a_1+a_2)\Gamma(-m+a_1+a_3)\Gamma(-(2k-l-m-(n-1))+a_1+a_2)} \times NVT\ .
\]
Let us exhibit three nonnegative integers $a_1,a_2,a_3$ such that 
\begin{equation}\label{pasdevanne}
\begin{split}
&a_2+a_3\geq l+1\\ &a_1+a_3\geq m+1\\ &a_1+a_2> 2k-l-m-(n-1)\\ &a_1+a_2+a_3\leq k
\end{split}
\end{equation}
because for such a choice, $I_{\boldsymbol \alpha}(a_1,a_2,a_3)\neq 0$.
Recall that $k>l$ and $k>m$. Now the condition $l+m+(n-1)\leq 2k$ can be rewritten as $(k-l)+(k-m)\geq n-1$. Hence there exist integers  $p,q\geq 1$ such that 
\[k-l-p\geq 0,\qquad  k-m-q\geq 0, \qquad p+q=(n-1)\ .\]
As $n\geq 4$, either $p$ or $q$ is greater than or equal to $2$. Up to a permutation of the indices, we may assume that $p\geq 2$. Now let
\[ a_1=k-l-p+1,\quad a_2 = k-m-q,\quad a_3 = m+l-k+(n-1)-1
\]
The three coefficients are $\geq 0$. Moreover
\begin{equation*} 
\begin{split}&a_2+a_3 = l+(n-1)-q-1\\&a_1+a_3 = m+(n-1)-p \\&a_1+a_2 =  2k-l-m-(n-1)+1 \\&a_1+a_2+a_3=k
\end{split}
\end{equation*}
and conditions \eqref{pasdevanne} are satisfied as $(n-1)-q-1 = p-1\geq 1$ and $(n-1)-p=q\geq 1$.
 \end{proof}

In \cite{co}, the following result was obtained.

\begin{theorem}[generic multiplicity 1 theorem]\label{genuniq}
 Assume that  $\boldsymbol \lambda\in \mathbb C^3$ is not a pole. Then $K^{\boldsymbol \lambda}$ is, up to a scalar, the unique $\boldsymbol \lambda$-invariant distribution on $S\times S\times S$.
\end{theorem}

In other words, when $\boldsymbol \lambda$ is not a pole, $\dim Tri(\boldsymbol \lambda)=1$ or, more precisely, $Tri(\boldsymbol \lambda) = \mathbb C \widetilde K^{\boldsymbol \lambda}$. 

The proof of this theorem used two results, which we now recall , because they will be needed in the sequel.

\begin{lemma}\label{invsing1}
 Let $\boldsymbol \alpha\in \mathbb C^3$ and $\boldsymbol \lambda$ the associated spectral parameter. Let $\mathcal O$ be a $G$-invariant open set of $S\times S \times S$ which contains $\mathcal O_3$ as a  closed submanifold. Let $T\neq 0$ be a distribution in $\mathcal O$, $\boldsymbol \lambda$-invariant and such that $Supp(T) \subset \mathcal O_3$. Then $\alpha_3 = -(n-1)-2k$ for some $k\in \mathbb N$. Moreover, if  $S$ is a $\boldsymbol \lambda$-invariant distribution in $\mathcal O$  such that $Supp(S) \subset \mathcal O_3$, then $S$ is proportional to $T$.
 \end{lemma}
 
 \begin{lemma}\label{invsing2}
 Let $\boldsymbol \alpha\in \mathbb C^3$ and $\boldsymbol \lambda$ the associated spectral parameter. Let $T\neq 0$ be a distribution on $S\times S\times S$ which is $\boldsymbol \lambda$-invariant and such that $Supp(T)\subset \mathcal O_4$. Then $\alpha_1+\alpha_2+\alpha_3 = -2(n-1)-2k$ for some $k\in \mathbb N$. 
\end{lemma}

The proof of both statements is obtained by standard application of Bruhat's theory, and is essentially contained in \cite{co}. The only exception is the second statement of the first lemma, the multiplicity one result, but it is only a refinement and its proof can be considered as routine. Notice that there is no multiplicity 1 assertion in the second lemma, a point which reflects a basic difference between poles of type I and poles of type II.

\section {Distributions smoothly supported on a submanifold}\label{smoothsupp}

 Let $E$ be a real vector space of dimension $n$, and $F$ a subspace of codimension $p$.  Choose  coordinates $(x_1,x_2,\dots,x_p,y_{p+1}\dots, y_n)$ such that 
 \[F= \{ (x,y)\in \mathbb R^p\times \mathbb R^q,\  x_1=x_2=x_p=0\}\ .
 \]
For $J=(j_1, j_2,\dots, j_p)$ a multi-index, let 
 \[x^J = x_1^{j_1}\dots x_p^{j_p},\quad \partial^J = \Big(\frac{\partial}{\partial x_1}\Big)^{j_1} \dots  \Big(\frac{\partial}{\partial x_p}\Big)^{j_p}\ .\]
 
 Let $U$ be an open set in $E$. A (smooth) \emph{transverse differential operator} on $U$ is an operator from $\mathcal C^\infty(U)$ into $\mathcal C^\infty(F\cap U)$ given by
 \[Df (y) = \sum_{J} a_J(y)\, \partial^J\!f \,(y)\ ,\quad f\in \mathcal C^\infty(U),\  y\in F\cap U
 \]
 where $J$ runs through all multi-indices and $(a_J)$ is a locally finite family of smooth functions on $F\cap U$. Notice that the functions $a_J$ are well determined, as can be observed by testing the operator $D$ against the functions $x^J$. A transverse differential operator $D$ is said to be \emph{of transverse order $\leq m$}  if $a_J\equiv 0$ for all multi-indices $J$ such that $\vert J\vert = j_1+j_2+\dots+j_p>m$.
  
 Let $E^*$ be the dual space of $E$, and let $F^\perp = \{ \psi \in E^*, \psi_{\vert E} =0\}$. The space $F^\perp$ is canonically isomorphic to $(E/F)^*$. The coordinate forms $\xi_j : (x,y) \mapsto x_j,  1\leq j\leq p$ form a basis of $F^\perp$.
 
 Let $D = \sum_{\vert J\vert\leq m}a_J \partial^J$ be a transverse differential operator on $U$ of transverse order $\leq m$. Then its \emph{symbol} $\sigma_m(D)$ is  defined by
 \[\sigma_m(D)( y,\xi) = \sum_{ \vert J\vert =m} a_J(y) \xi^J\ ,
 \] 
 where $y\in F\cap U$ and $\xi\in F^\perp$. To see the intrinsic  character of the symbol, there is a useful formula. For $y\in F\cap U$ and $\xi\in F^\perp$, choose smooth functions $f,\varphi$ on $U$ such that $f(y) = 1$ and $d\varphi(y) = \xi$. Then
 \begin{equation}\label{intrinsic}
 \sigma(D)(y,\xi) = \lim_{t\rightarrow \infty} t^{-m} e^{-t\varphi} D(e^{t\varphi} f)(y)\ .
 \end{equation}

Fix a Lebesgue measure $dy$ on $F$. To any transverse differential operator $D$ is associated the distribution $T_D$ given by
 \[T_D(f) = \int_{U\cap F} Df(y) dy
 \]
 for any test function $f\in \mathcal C_c^\infty(U)$.
 
 It is possible to characterize the distributions which are of this form. For $T$ a distribution on $U$, denote by $WF(T)$ the \emph{wave front set} of $T$  (see \cite{h} ch. VIII).

\begin{lemma}\label{smoothtrans}
 Let $T$ be a distribution on $U$, with $Supp(T)\subset F\cap U$. Then the two following properties are equivalent :

$i)$ $WF(T) \subset (F\cap U)\times F^\perp$

$ii)$ There exists a unique transverse differential operator $D$ on $F\cap U$  such that 
\[(T,\varphi) = \int_F D\varphi(y) dy\ .
\]

\end{lemma}
For a proof see \cite{bc}. A distribution $T$ which satisfy the conditions of  the Lemma is said to be \emph{smoothly supported} on $F$.

An important source of smoothly supported distribution is obtained by the following result, which is a version in the present context  of the principle "invariance implies smoothness".

\begin{lemma}\label{invsmooth}
 Let $X_j, 1\leq j\leq m$ be a family of smooth vector fields on $U$ such that
\[\forall y\in F\cap U,\quad \forall j, 1\leq j\leq m, \ X_j(y)\in F \text{ and } \{X_j(y), 1\leq j\leq m\} \text{ generate } F \ .\]
Let $(a_j)_{1\leq j\leq m}$ be a family of smooth functions on $E$ and $(U_j)_{1\leq j\leq m}$ a family of distributions on $U$ smoothly supported in $F\cap U$. Let $T$ be a distribution on $U$ which satisfies
\smallskip

$i)$ $Supp\, (T)\subset F\cap U$
\smallskip

$ii)$  $\forall j,1\leq j\leq m$, $(X_j+a_j) T=U_j$.
\smallskip

\noindent
Then $T$ is smoothly supported on $F\cap U$.
\end{lemma}
\begin{proof} Using Lemma \ref{smoothtrans}, it suffices to prove that $WF(T)\subset (F\cap U)\times F^\perp$. Now for $D$ any differential operator on $U$,
\[WF(T) \subset Char(D)\cup WF(DT)\ ,
\]
where $Char(D)\subset U\times E^*$ is the \emph{characteristic set} of $D$ (see \cite{h} Theorem 8.3.1).
Apply to $D = X_j +a_j$. Observe that $WF(U_j) \subset (F\cap U)\times F^\perp$ by assumption. Next $Supp(T)\subset F\cap U$, and for $x\in F\cap U, \xi \in E^*$, \[(x,\xi)\in Char(X_j+a_j)\Longleftrightarrow \xi(X_j(x))=0\ .\]

As the $X_j(x)$ generate $F$, it follows that $(x,\xi)\in WF(T)$ implies $\xi\in F^\perp$ and the statement follows.
\end{proof}

Let $E'$ be another vector space, also of dimension $n$ and $F'$ a subspace of  $E'$ also of codimension $p$. Let $\Phi$ be a diffeomorphism on an open subset $U'$ of $E'$ which satisfies
 $\Phi(U\cap F) \subset U'\cap F'$. Let $ \Phi^* f = f\circ \Phi $. This formula defines isomorphisms between  $ \mathcal C^\infty(U')$ (resp.  $\mathcal C^\infty(U'\cap F')$) and $ \mathcal C^\infty(U)$ (resp $\mathcal C^\infty(U\cap F)$). 
 If $D$ is a transverse differential operator on $U$ of order $\leq m$, let 
 \[\Phi_*D = {\Phi^*}^{-1}\circ D\circ \Phi^*\ .\] 
 Then $\Phi_*D$  is a transverse differential operator on $U'$ (relative to $F'$) of transverse order $\leq m$. The transverse symbols of $D$ and $\Phi_* D$   are related by the formula \begin{equation}\label{symbdiffeo}
\sigma_m(\Phi_ *D) (y,\xi)  = \Phi_* \,\sigma_m(D) (y,\xi) = \sigma_m(D)\big(\Phi^{-1} (y), \xi\circ d\Phi^{-1}(y)^{-1}\big)\ .
\end{equation}
The formula is easily obtained from \eqref{intrinsic}.

Let $T$ be a distribution on $U$ smoothly supported on $F$, and let $D_T$ be the associated transverse differential operator. Let $\Phi_*T$ be the distribution on $U'$ defined by
\[ (\Phi_* T, \varphi) = (T, \varphi\circ \Phi), \qquad \varphi\in\mathcal C_c^\infty(U') \]
Then $\Phi_*T$ is smoothly supported on $U'\cap F'$. Fix a Lebesgue  measure $dx'$ on $F'$. Then there exists an associated  transverse differential operator $D_{\Phi_*T}$ on $U'\cap F$.

\begin{lemma} The transverse differential operator $D_{\Phi_*T}$ is given by 
\begin{equation}\label{transsymbol}
D_{\Phi_*T} = \vert \det \big( (D\Phi^{-1})_{\vert F'\rightarrow F}\big)\vert\, \Phi_*D_T
\end{equation}

\end{lemma}
\begin{proof} By definition,
\[(\Phi_*T,f) = \int_{F\cap U} D(f\circ \Phi)(y) dy = \int_{F\cap U'} D(f\circ \Phi)\circ \Phi^{-1}(y')\,\vert \det \big( (D\Phi^{-1}(y'))_{\vert F}\big)\vert\ \,dy'
\]
by the change of variables $y'=\Phi(y)$. The result follows.
\end{proof}
For $T$ a distribution smoothly supported on $U\cap F$, which is of transverse order less than $m$, define its \emph{ transverse symbol} $\sigma_m(T)$ by
\[\sigma_m(T)(x,\xi) = \sigma_m(D_T)(x,\xi)\, dx\ .
\]
These notions can be generalized to the setting of a regular submanifold $N$ of a manifold $M$  to get definitions of a  transverse differential operator on $N$,  of a distribution on $M$ smoothly supported on $N$ and of its transverse symbol (after the choice of a smooth measure is chosen on $N$). In particular, the transverse symbol of a differential operator of order $\leq m$ is a section of the bundle $S^m\mathcal N$ ($m$-symmetric tensor product of the conormal bundle on $N$), whereas the transverse symbol of a distribution on $M$ smoothly supported in $N$ and of transverse order $\leq m$ is a section of the bundle $S^m\mathcal N \otimes \vert \Lambda\vert$, where $\vert\Lambda\vert$ is the density  bundle of $N$ (this is the reason to add "$dx$" in the definition of the transverse symbol of a distribution). We omit details (see \cite{kv} for related ideas).

The main properties of the symbol map are summarized in the next proposition.

\begin{proposition}\label{symbol} Let $N$ be a regular submanifold of a manifold $M$. Let $T$ be a distribution smoothly supported in $N$, of transverse order $\leq m$. Let $\sigma_m(D)$ be the transverse symbol of $T$.
\smallskip
 
 $i)$ let $\Phi$ be a diffeormorphism of $M$, such that $\Phi(N)=N$. Then $\Phi_*T$ is smoothly supported on $N$, of transverse order less than or equal to $m$ and its transverse symbol is given by
 \[\sigma_m(\Phi_* T)(x,\xi) = \big\vert \det \big(D\Phi^{-1}_{\vert T_xN}(x)\big)\big\vert \sigma_m(\Phi^{-1}(x), \xi\circ D\Phi^{-1}(x)^{-1}\big)\ .
 \]
 
 $ii)$ Let $a$ be a smooth function on $M$. Then $aT$ is smoothly supported on $N$ and is of transverse order $\leq m$. Its transverse symbol satisfies
$\sigma_m(aT) = a \,\sigma_m(T)$.
\smallskip

$iii)$ if $\sigma_T\equiv 0$, then $T$ is of transverse order  $\leq m-1$.
\end{proposition}

\section{The theorem of the support}

Let $\boldsymbol \lambda\in \mathbb C^3$ be a pole and assume that $\boldsymbol \lambda\notin Z$. The distribution $\widetilde {\mathcal K}^{\boldsymbol \lambda}$ is not $0$ (by the very definition of $Z$), and is singular as stated in Proposition \ref{suppK}. The following result is a major step towards proving Theorem \ref{mult1}. 
\begin{theorem}\label{singsupp}
Let $\boldsymbol \lambda$ be a pole and assume $\boldsymbol \lambda\notin Z$. Let $T\neq 0$ be a $\boldsymbol \lambda$-invariant distribution on $S\times S\times S$. Then
\[Supp(T) = Supp(\widetilde {\mathcal K}^{\boldsymbol \lambda})\ .
\]
\end{theorem}

Recall that $\mathcal K_{\boldsymbol \alpha,\, \mathcal O_0}$ is the distribution on $\mathcal O_0$ defined by integration against $\vert x-y\vert^{\alpha_3}\vert y-z\vert^{\alpha_1}\vert z-x\vert^{\alpha_2}$. The distribution $\mathcal K_{\boldsymbol \alpha,\, \mathcal O_0}$ is $\boldsymbol \lambda$-invariant.
The most difficult step towards proving Theorem \ref{singsupp} is the following result.

\begin{theorem}\label{noext}
Let $\boldsymbol \alpha$ be a pole  and assume that $\boldsymbol \alpha\notin Z$. Let $\boldsymbol \lambda$ be the associated spectral parameter. Then the distribution $\mathcal K_{\boldsymbol \alpha,\, \mathcal O_0}$ cannot be extended to $S\times S\times S$ as a $\boldsymbol \lambda$-invariant distribution. 
\end{theorem}

As the proof is long, let us sketch the main steps.
\smallskip

$\bullet$ The distribution $\mathcal K_{\boldsymbol \alpha,\, \mathcal O_0}$ has a "natural" extension (say $F$) to $S\times S\times S$, given by the degree $0$ coefficient in the Laurent expansion of the meromorphic function $\boldsymbol \alpha \longmapsto \mathcal K_{\boldsymbol \alpha}$ restricted to a (well chosen) complex line in $\mathbb C^3$.

$\bullet$ The distribution $F$ is not $\boldsymbol \lambda$-invariant, but "almost" invariant, in the sense that, for each $g\in G$,   \[F\circ \boldsymbol \pi_{\boldsymbol \lambda}(g)-F= E_g\]
where $E_g$ is a distribution (depending on $g$) which is identically $0$ on $\mathcal O_0$.

$\bullet$ If $T$ is a $\boldsymbol \lambda$-invariant distribution extending
$\mathcal K_{\boldsymbol \alpha,\, \mathcal O_0}$, then $S=F-T$ vanishes identically on $\mathcal O_0$ and  satisfies 
\begin{equation}\label{functS}
S\circ \boldsymbol \pi_{\boldsymbol \lambda}(g)-S= E_g\ .
\end{equation}
\noindent

$\bullet$ Both $S$ and $E_g$ are supported in $\mathcal O_1\cup \mathcal O_2\cup \mathcal O_3\cup \mathcal O_4$. Successively for each orbit $\mathcal O_j$, it is possible to define and compute the transverse symbols of both sides of \eqref{functS}. In $\mathcal O_j$, there is a point which is fixed by all $a_t, t\in \mathbb R$. Evaluate both symbols at  this point. Studying their behavior as $t$ varies eventually leads to a contradiction.
\smallskip

Three separate cases  have to be considered : the case of a generic pole of type I (which automatically is not in $Z$), the case of a generic pole of type II not in $Z$, and the case of a pole of type I+II not in $Z$. The proofs rely on the same ideas, but are formally different. Details are given for the first case, we ask the the friendly reader to accept the more sketchy treatment of the last two cases. 

\subsection{Type I}
 Assume  $\boldsymbol \alpha=(\alpha_1,\alpha_2, \alpha_3)$  is a generic pole (recall Definition \ref{generic}) of type I$_3$, that is $\alpha_3=-(n-1)-2k$ for some $k\in \mathbb N$. The assumption of genericity implies that $\boldsymbol \alpha\notin Z$. It also implies that among the four $\Gamma$ factors in the holomorphic normalization of $K_{\boldsymbol \alpha}$ only one (namely $\Gamma(\frac{1}{2} \alpha_3+\rho)$) becomes singular. Let $s$ be a complex parameter, $s\neq 0$, $\vert s\vert$ small. Let 
\[\boldsymbol \alpha(s) = (\alpha_1,\alpha_2, \alpha_3+2s), \quad \boldsymbol \lambda(s) = (\lambda_1+s, \lambda_2+s, \lambda_3)\]
and let 
\[\mathcal F(s)= \frac{(-1)^k}{k!}\, \Gamma(\frac{\alpha_1}{2} +\rho)\, \Gamma(\frac{\alpha_2}{2} +\rho)\, \Gamma(\frac{\alpha_1+\alpha_2+\alpha_3}{2} +2\rho+s)\, \widetilde {\mathcal K}_{\boldsymbol \alpha(s)}.\] 

This is well defined distribution-valued function, at least for $\vert s\vert$ small. Consider its Taylor expansion at $0$
 \begin{equation}\label{Fs}
\mathcal F(s) =   F_0+ s F_1 +O(s^2) ,
\end{equation}
where $F_0, F_1$ are distributions on $S\times S\times S$.
The observation on the normalization factor implies that the distribution $ F_0$ is a non zero multiple of $\widetilde {\mathcal K}_{\boldsymbol \alpha}$, which, as $\boldsymbol \alpha \notin Z$, is  not equal to $0$. Hence $ F_0\neq 0$. Moreover  $ F_0$ is $\boldsymbol \lambda$-invariant and $Supp( F_0)\subset \overline{\mathcal O_3}$.

\begin{lemma} \ 

$i)$ The restriction of $ F_1$ to $\mathcal O_0$ is equal to $\mathcal K_{\boldsymbol \alpha, \, \mathcal O_0}$.

$ii)$ for any $g\in G$, 
\begin{equation}\label{F1}
 F_1 \circ \boldsymbol \pi_{\boldsymbol \lambda} (g) - F_1 = -(\ln \kappa(g^{-1},x)\kappa(g^{-1},y)) F_0
\end{equation}

$iii)$ for any $X\in \mathfrak p \simeq E$
\begin{equation}\label{F1inf}
F_1\circ d\boldsymbol \pi_{\boldsymbol \lambda}(X)   =-\langle X, x+y\rangle\,  F_0  \ .
\end{equation}

\end{lemma}

\begin{proof} 
Let $\varphi$ be a test function supported in $\mathcal O_0$. For $s\neq 0$, replace $\widetilde {\mathcal K}_{\boldsymbol \alpha(s)}$ by its expression in term of $\mathcal K_{\boldsymbol \alpha(s)}$ to get
\[\mathcal F(s) (\varphi) = \frac{(-1)^k}{k!} \frac{1}{\Gamma( -k+s)}\, {\mathcal K}_{\boldsymbol\alpha(s), \mathcal O_0}(\varphi)
\ .\]
Let $s \rightarrow 0$ to get  $F_0(\varphi) = \mathcal K_{\boldsymbol \alpha, \mathcal O_0}(\varphi)$ and $i)$ follows.

Next, for $f\in \mathcal C^\infty(S)$ and arbitrary $\lambda\in \mathbb C$,
\[\big(\pi_{\lambda+s} (g)f\big)(x) = \kappa(g^{-1},x)^{s} \big(\pi_{\lambda}(g) f\big)(x)
\]
and hence for $f\in \mathcal C^\infty(S\times S\times S)$
\[\big(\boldsymbol \pi_{\boldsymbol \lambda(s)}(g)f \big)(x,y,z) = 
\kappa(g^{-1},x)^s\, \kappa(g^{-1},y)^{s} \big(\boldsymbol \pi_{\boldsymbol \lambda}(g)f\big) (x,y,z)
\]
\[= \big(1+s \ln \big(\kappa(g^{-1},x) \kappa(g^{-1},y)\big)+O(s^2)\big)\,\big(\pi_{\boldsymbol \lambda}(g)f\big) (x,y,z)\ .
\]
Recall that $\mathcal F(s)\circ \boldsymbol \pi_{\boldsymbol \lambda(s)}(g) = \mathcal F(s)$. Compare the Taylor expansion of both sides to obtain
\begin{equation} \big(F_1 + \ln\big(\kappa(g^{-1},x) \,\kappa(g^{-1},y)\big)  F_0\big)\circ \boldsymbol \pi_{\boldsymbol \lambda}(g)  = F_1 \ ,
\end{equation}
and use  the $\boldsymbol \lambda$-invariance of $F_0$  to get $ii)$. Notice that, as the conformal factor of an element in $K$ (a rotation) is identically equal to $1$, $ii)$ implies that $ F_0$ is $K$-invariant.

For $iii)$,  let $X\in \mathfrak p$. Then $d\boldsymbol \pi_{\boldsymbol \lambda}(x)F_0=0$ by the $\boldsymbol \lambda$-invariance of $F_0$, and using \eqref{derkappa}
\[\ln \big(\kappa(\exp-tX,x)\kappa(\exp-tX,y)\big) =\ln \big(1+t(\langle X,x\rangle +\langle X,y\rangle)+O(t^2)\big)
\]
\[ = t\langle X,x+y\rangle +O(t^2)\ .
\]
Let $g=a_t$ in \eqref{F1} and take the derivative of both sides at $t=0$ to get 
\[ F_1\circ d\boldsymbol \pi_{\boldsymbol \lambda} (X)  = -\langle X,x+y\rangle\,F_0\ .
\]

\end{proof}

 Now, by assumption $\alpha_1,\alpha_2\notin -(n-1)-2\mathbb N$. Hence, on \[\mathcal O_3' = S\times S\times S\setminus\overline{\mathcal O_3} = \{ (x,y,z)\in S\times S\times S, x\neq y\}
\] 
it is possible to extend the distribution $\mathcal K_{\boldsymbol \alpha, \,\mathcal O}$ by meromorphic continuation, to get a distribution on $\mathcal O_3'$ (denoted by $\mathcal K_{\boldsymbol \alpha,\,\mathcal O_3'}$) which is $\boldsymbol \lambda$-invariant. Moreover, it is the only extension to $\mathcal O_3'$ which is $\boldsymbol \lambda$-invariant. In fact, if there is one such, say $ T'$, then the difference $ T'-{\mathcal K}_{\boldsymbol \alpha,\mathcal O_3'}$ is supported in $\mathcal O_3'\setminus \mathcal O_0=\mathcal O_1\cup \mathcal O_2$ and $\boldsymbol \lambda$-invariant. Now use twice Lemma \ref{invsing1} to conclude that this distribution has to be $0$ : once for $\mathcal O = \mathcal O_0\cup \mathcal O_1$ containing $\mathcal O_1$, the second for $\mathcal O = \mathcal O_0\cup \mathcal O_2$containing $\mathcal O_2$. On the other hand, the restriction to $\mathcal O_3'$ of $ F_0$  is $0$, and hence, by  \eqref{F1},  the restriction of $ F_1$ to $\mathcal O_3'$ is $\boldsymbol \lambda$-invariant. Hence $F_1$ coincides on $\mathcal O'_3$ with $\mathcal K_{\boldsymbol \alpha,\,\mathcal O_3'}$.
We now are in position to prove the following result which implies {\it a fortiori} Theorem \ref{noextI} for the case of a generic pole of type I.

\begin{proposition}\label{noextI}
There is no  $\boldsymbol \lambda$-invariant distribution on $\mathcal O_4' = S\times S\times S\setminus \mathcal O_4$ which extends $\mathcal K_{\boldsymbol \alpha,\mathcal O'_3}$. 
\end{proposition}

\begin{proof}
For convenience, when $F$ is a distribution on $S\times S\times S$, denote by $F'$ its restriction to $\mathcal O_4'$. Assume $T$ is a distribution on $\mathcal O'_4$ which extends $\mathcal K_{\boldsymbol \alpha,\,\mathcal O'_3}$ and is $\boldsymbol \lambda$-invariant.   Let $ S = T- F_1'$. Then $S$ is supported in $\mathcal O_3$, which is a closed regular submanifold of $\mathcal O_4'$. Let $\mathfrak g$ be the Lie algebra of $G$, and for $X\in \mathfrak g$, let $\widetilde X$ be the vector field on $\mathcal O_4'$ induced by the one-parameter subgroup $\exp tX, t\in \mathbb R$. Recall that $G$ acts transitively on $\mathcal O_3$, so that for any $m\in \mathcal  O_3$  the vector space generated by $\{ \widetilde X(m), X\in \mathfrak g\}$ is equal to $T_m\mathcal O_3$, the tangent space  to $\mathcal O_3$ at $m$.

As both $T$ and $F_1'$ are $K$-invariant, 
\begin{equation}\label{KinvS}
 \forall X\in \mathfrak k,\hskip 2cm S\circ d\boldsymbol \pi_{\boldsymbol \lambda} (X)=0\ .
 \end{equation}
Now for $X\in \mathfrak p$,
\begin{equation*}
d\boldsymbol \pi_{\boldsymbol \lambda} (X)(x,y,z)  = -\widetilde X(x,y,z) + (\lambda_1+\rho)\langle X,x\rangle +(\lambda_2+\rho)\langle X,y\rangle+(\lambda_3+\rho)\langle X,z\rangle
\end{equation*}
is a first order differential operator of the form $-\widetilde X+a_X$ with $a_X$ a smooth function on $S\times S\times S$. Moreover
$ T\circ  d\boldsymbol \pi_{\boldsymbol \lambda} (X) = 0$
as $ T$ is assumed to be $\boldsymbol \lambda$-invariant. Restrict \eqref{F1inf} to $\mathcal O_4'$ to get
$F'_1 \circ d\boldsymbol \pi_{\boldsymbol \lambda} (X)= - \langle X, x+y\rangle F'_0$. Hence
\begin{equation}\label{pinvS}
\forall X\in \mathfrak p,\hskip 2cm S\circ d\boldsymbol \pi_{\boldsymbol \lambda}(X) = \langle X,x+y\rangle F_0'\ .
\end{equation}
 As $F_0'$ is invariant under $\boldsymbol \pi_{\boldsymbol \lambda}$, $F_0'$ is smoothly supported in $\mathcal O_3$. Take into account \eqref{KinvS} and \eqref{pinvS}, use Lemma \ref{invsmooth} to conclude that $S$ is smoothly supported in $\mathcal O_3$. 

For any $g\in G$,  \eqref{F1} and the $\boldsymbol \lambda$-invariance of $T$ imply
\begin{equation}\label{quasiS}
 S\circ \boldsymbol \pi_{\boldsymbol \lambda}(g)- S = -\ln \big(\kappa(g^{-1},x) \kappa(g^{-1},y)\big)  F'_0\ .
\end{equation}
The next elementary result is needed in order to make  connection with the notation and results of section \ref{smoothsupp}.

\begin{lemma} Let $F$ be a distribution on a $G$-invariant open set in $S\times S\times S$. For any $g\in G$,
\begin{equation}\label{symbF}
 F \circ \boldsymbol \pi_{\boldsymbol \lambda}(g) = a_g\, \big(g^{-1}\big)_* \,F
\end{equation}
where $a_g$ is given by
\[a_g(x,y,z) =\kappa(g,x)^{-(\lambda_1+\rho)}  \kappa(g,y)^{-(\lambda_2+\rho)}  \kappa(g,z)^{-(\lambda_3+\rho)} \ .
\]
\end{lemma}
\begin{proof}
Let $\varphi\in \mathcal C^\infty(S)$. As $ \kappa(g^{-1},x) = \kappa(g,g^{-1}x)^{-1}$,  
\[\boldsymbol \pi_ \lambda(g)\varphi(x) = \kappa\big(g^{-1}, x)^{\lambda+\rho} \varphi(g^{-1}(x)\big) = \big(\kappa(g, .)^{-1}\varphi\big)(g^{-1}(x))\ .\]
Hence for $f\in \mathcal C^\infty(S\times S\times S)$,
\[\boldsymbol \pi_{\boldsymbol \lambda}(g) f(x) = (a_g\,f)(g^{-1}(x))\]
 which amounts to $\boldsymbol \pi_{\boldsymbol \lambda}(g)f =  (g^{-1})^* (a_gf)$. Hence
 \[\big(F\circ \boldsymbol \pi_{\boldsymbol \lambda}(g)\big) f= F(\boldsymbol \pi_{\boldsymbol \lambda}(g)f) = F((g^{-1})^* (a_gf))= (a_g (g^{-1})_*F)(f)\ .
 \]
\end{proof}

From the computation of the residue of $\mathcal K^{\boldsymbol \lambda}$ at a generic pole of type I (see \cite{bc} Theorem 2.2),  it is known that $ F_0$ has a (global) transverse order along $\mathcal O_3$ equal to $2k$. From \eqref{quasiS} and \eqref{symbF} follows that $ S$ also has a global transversal order along $\mathcal O_3$, say $m$, and $m\geq 2k$.

 At a point $p=(x,x,z)\in \mathcal O_3$ with $x\neq z$, the tangent space to $\mathcal O_3$  is 
\[T_p \mathcal O_3=\{  (u,u, v), u\in T_xS, v \in T_zS, \}\ ,
\]
 the conormal space (as a subspace of the cotangent space) is given by
\[N_p =\{ n_\xi= (\xi, -\xi, 0), \xi\in T_x^*S\}
\]
Let $g\in G$  such that  $g(p)=p$. The Jacobian of the differential $Dg(p)$ on $T_p\mathcal O_3$ is equal to $\kappa(g,x)^{n-1}\kappa(g,z)^{n-1}$. So, using the transformation rules  \eqref{symbdiffeo} and \eqref{transsymbol} and Lemma \ref{symbol}, \eqref{symbF} leads to

\[\sigma_m( S\circ \boldsymbol \pi_{\boldsymbol \lambda}(g))\big((x,x,z),n_\xi\big) \] \[=\kappa(g,x)^{-(\lambda_1+\lambda_2)} \kappa(g,z)^{-\lambda_3+\rho}
\sigma_m(S)\big((x,x,z), n_{\xi\circ Dg(x)^{-1}}\big)\ .
\]

Apply this to  $x=\mathbf 1, z=-\mathbf 1$, and $g=a_t$. Then $a_t(\mathbf1)=\mathbf 1, a_t(-\mathbf 1) = -\bf 1$, $Da_t(\mathbf1) = e^{-t} \Id, Da_t(-\mathbf 1) = e^t \Id$. Recall that $\lambda_1+\lambda_2-\lambda_3 +\rho=-2k$. So, after some computation, \eqref{quasiS} implies
\begin{equation}\label{quasihom}
\big(e^{t(2k-m)} -1\big) \sigma_m( S)\big((\mathbf 1,\mathbf1,-\mathbf1), n_\xi \big) =
 -t \,\sigma_m(F_0)\big((\mathbf 1,\mathbf 1,-\mathbf 1), n_\xi\big)
\end{equation}
If $m>2k$, the right hand side is $0$, which forces $\sigma_m( S)\big(\mathbf 1,\mathbf1,-\mathbf1) =0$ by and hence $\sigma_m(S) =0$ by \eqref{quasiS}, and hence $S$ is of transverse order $\leq m-1$, a contradiction. So the only possibility is  $m=2k$. But then, the left hand side is $0$. So $\sigma_{2k}(F_0, (\mathbf 1,\mathbf1,-\mathbf1))=0 $ and be the invariance of $F_0$, $\sigma_{2k} (F_0) =0$ on all of $\mathcal O_3$. Hence \eqref{quasihom} leads to a contradiction, thus proving Proposition \ref{noextI}.
\end{proof}

The proof of Theorem \ref{singsupp} in the case of a generic pole of type I is now easy. Let $ T$ be a distribution on $S\times S\times S$ which is $\boldsymbol \lambda$-invariant. Consider its restriction $T_{\mathcal O_0}$ to $\mathcal O_0$. As $\mathcal O_0$ is a unique orbit under $G$,  $\mathcal T_{\mathcal O_0}$ has to be a multiple of $\mathcal K_{\boldsymbol \alpha,\, \mathcal O_0}$. By Proposition \ref{noextI}, $ T_{\mathcal O_0}$ has to be $0$. Hence $ T$ is supported in $\mathcal O_1\cup \mathcal O_2\cup\mathcal O_3\cup \mathcal O_4$. As $\alpha_1,\alpha_2\notin -(n-1)-2\mathbb N$ by assumption, two applications of Lemma \ref{invsing1} show that $T$ is supported in $\overline{\mathcal O_3}$, which is the statement to be proved.

\subsection{Type II}

\begin{proposition} Let $\boldsymbol \alpha$ be a generic pole of II, and assume that $\boldsymbol \alpha$ is not in $Z$. Let $\boldsymbol \lambda$ be the associated spectral parameter. Then the distribution $ \mathcal K_{ \boldsymbol \alpha,\, \mathcal O_0}$ cannot be extended to a $\boldsymbol \lambda$-invariant distribution on $S\times S\times S$.
\end{proposition}
\begin{proof} Let $\boldsymbol \alpha =(\alpha_1,\alpha_2,\alpha_3)$ be a generic pole of type II, i.e.  $\alpha_1+\alpha_2+\alpha_3 = -2(n-1)-2k$ for some $k\in \mathbb N$, which moreover is not in $Z$. 

The distribution ${\mathcal K}_{\boldsymbol \alpha, \mathcal O_0}$ can be extended in a unique way to a $\boldsymbol \lambda$-invariant distribution on $\mathcal O'_4$. The proof is similar to the argument given in the proof of Proposition \ref{noextI}, as by assumption,  $\alpha_j\notin -(n-1)-2\mathbb N$ for $j=1,2$ or $3$. We skip details. Denote by  $\mathcal K_{\boldsymbol \alpha, \mathcal O_4'}$ the extension.

Now, again by assumption, among the $\Gamma$-factors involved in the normalization of $\mathcal K_{\alpha}$ only one is singular at $\boldsymbol \alpha$, namely $\Gamma(\frac{\alpha_1+\alpha_2+\alpha_3}{2}+2\rho)$.
For $s$ a complex parameter, $s\neq 0, \vert s\vert$ small, define
\[\boldsymbol \alpha(s) = (\alpha_1+\frac{2}{3}s, \alpha_2+\frac{2}{3}s,\alpha_3+\frac{2}{3}s),\  \boldsymbol \lambda(s)= (\lambda_1+\frac{2}{3}s, \lambda_2+\frac{2}{3}s, \lambda_3+\frac{2}{3}s)\  .
\]
Notice  that $\alpha_1(s)+\alpha_2(s)+\alpha_3(s) = \alpha_1+\alpha_2+\alpha_3+2s$.
For $s\in \mathbb C$, $s\neq 0$ and $\vert s\vert$ small, let
 \[
\mathcal F(s)= \frac{ (-1)^k}{ k!}\, \Gamma\Big(\frac{\alpha_1(s)}{2}+\rho\Big)\,\Gamma\Big(\frac{\alpha_2(s)}{2}+\rho\Big)\,\Gamma\Big(\frac{\alpha_3(s)}{2}+\rho\Big)\,\mathcal K_{\boldsymbol \alpha(s)}
\]

 The Taylor expansion of  $\mathcal F(s)$ at $s=0$ reads  
 \[\mathcal F(s) =   F_0 +s F_1 + O(s)
\]
 where $ F_0,  F_1$ are distributions on $S\times S\times S$.

The distribution $ F_0$ is a (non zero) multiple of $\widetilde  {\mathcal K}_{\boldsymbol \alpha}$, which is not $0$ as $\boldsymbol \alpha\notin Z$. Hence $ F_0$ is a (non zero) $\boldsymbol \lambda$-invariant distribution supported in $\mathcal O_4$.

\begin{proposition} The distribution $ F_1$ satisfies the following properties :

$i)$ on $\mathcal O_0$, $ F_1$ coincides with $\mathcal K_{\boldsymbol \alpha,\, \mathcal O_0}$.

$ii)$ for any $g\in G$,
\begin{equation}
 F_1 \circ \boldsymbol \pi_{\boldsymbol \lambda} (g) - F_1 = -\frac{2}{3} \ln \big(\kappa(g^{-1},x)\kappa(g^{-1}, y) \kappa(g^{-1},z)\big) F_0
\end{equation}
\end{proposition}

\begin{proof} For $s\neq 0$, and $\varphi\in \mathcal C^\infty(S\times S\times S) $ supported in $\mathcal O_0$, 
 \[\mathcal F_(s) (\varphi) = \frac{ (-1)^k}{ k!}\frac{1}{ \Gamma(-k+s)}\,\mathcal K_{\boldsymbol \alpha(s), \mathcal O_0}(\varphi)\]
Let $s\rightarrow 0$ to get $i)$. The proof of $ii)$ is similar to the proof of  \eqref{F1}.
\end{proof}
Assume now there exists   a distribution $ T$  on $S\times S\times S$, which extends $\mathcal K_{\boldsymbol \alpha, \mathcal O_0}$ and is $\boldsymbol \lambda$-invariant. As $Supp( F_0) \subset \mathcal O_4$, the restriction to $\mathcal O_4'$ of  $ F_0$ and $\mathcal K_{\boldsymbol \alpha, \mathcal O_4'}$ coincide. Let $ S=  T- F_0$. Then $ S$ is smoothly supported in $\mathcal O_4$ and satisfies
\begin{equation}\label{quasiS2}
 S\circ  \boldsymbol \pi_{\boldsymbol \lambda} (g)- S= -\frac{1}{2}(\ln \kappa(g^{-1},x)\kappa(g^{-1}, y) \kappa(g^{-1},z)\, F_0
\end{equation}
Let $(x,x,x)\in \mathcal O_4$. Then $T_{(x,x,x)}\mathcal O_4 = \{(u,u,u), u\in T_xS\}$ so that 
\begin{equation*}
N_{(x,x,x)}=\{ \boldsymbol \xi=(\xi_1,\xi_2,\xi_3),\  \xi_1+\xi_2+\xi_3 = 0,\  \xi_j \in T_x^*S, j=1,2,3\}\ .
\end{equation*}
From the computation of the residue of $\mathcal K_{\boldsymbol \alpha}$ at a pole of type II, the distribution $ F_0$ has global transverse order equal to $2k$ (see \cite{bc}). From  \eqref{quasiS2}, the distribution $ S$ has also a global order, say $m$, and $m\geq 2k$. Let $g\in G$ such that $g(x) = x$. Then
\begin{equation}
\begin{split}
&\sigma_m \big(S\circ \boldsymbol \pi_{\boldsymbol \lambda}(g)\big) \big((x,x,x),\boldsymbol \xi\big) =\\ 
&\kappa(g,x)^{-(\lambda_1+\lambda_2+\lambda_3 +\rho)} \sigma_m(S)\big((x,x,x), \boldsymbol \xi \circ Dg(x)^{-1}\big)\ .
\end{split}
\end{equation}

Let $x=\mathbf 1$, and let $g=a_t$ for $t\in \mathbb R$. Then $a_t(\mathbf 1) = \mathbf 1$and $Da_t(\mathbf 1) = e^{-t} \Id$, so that \eqref{quasiS} implies
\begin{equation}
(e^{t(2k-m)}-1)\, \sigma_m(S) \big((\mathbf 1,\mathbf 1,\mathbf 1), \boldsymbol \xi\big) = - t \,\sigma_m( F_0) \big((\mathbf 1,\mathbf 1,\mathbf 1), \boldsymbol \xi\big)
\end{equation}
An argument similar to the one used in the proof of Proposition \ref{noextI} yields a contradiction. 
\end{proof}

The proof of Theorem \ref{singsupp} in the case of a pole of type II (generic and not in $Z$) is similar to the proof given for the case of a pole of type I, and we omit it.

\subsection{Type I+II}

\begin{proposition}\label{noextI+II}
Let $\boldsymbol \alpha$ be a pole of type I+II, and assume moreover that 
$\boldsymbol \alpha \notin Z$.  Let $\boldsymbol \lambda$ be its associated spectral parameter. Then the distribution $\mathcal K_{\boldsymbol \alpha,\,\mathcal O_0}$ can not be extended to a $\boldsymbol \lambda$ -invariant distribution on $S\times S\times S$.
\end{proposition}

\begin{proof} 
Let $\boldsymbol \alpha = (\alpha_1,\alpha_2,\alpha_3)$ and assume that 
\[\alpha_1+\alpha_2+\alpha_3 = -2(n-1)-2k, \quad \alpha_3 = -(n-1)-2l
\]
for some $k,l\in \mathbb N$. Assume moreover that $\boldsymbol \alpha\notin Z$. This time, \emph{two} of the four $\Gamma$-factors (namely $\Gamma( \frac{\alpha_3}{2}+\rho)$ and $\Gamma(\frac{\alpha_1+\alpha_2+\alpha_3}{2}+2\rho)$) in the normalization of $\mathcal K_{\boldsymbol \alpha}$ are singular at $\boldsymbol \alpha$.

For $s\in \mathbb C$, let
\[\boldsymbol \alpha(s) = (\alpha_1,\alpha_2, \alpha_3+2s), \qquad \boldsymbol \lambda (s) = (\lambda_1+s, \lambda_2+s, \lambda_3)\ ,
\]
and consider the distribution-valued function
\begin{equation}\label{defF}
\mathcal F(s) = \frac{(-1)^{k+l}}{k!\,l!}\Gamma\Big(\frac{\alpha_1}{2}+\rho\Big) \Gamma\Big(\frac{\alpha_2}{2}+\rho\Big)\widetilde {\mathcal K}_{\boldsymbol \alpha(s)}\ .
\end{equation}
The Taylor expansion of $\mathcal F$ at $s=0$ reads
\begin{equation}\label{TaylorF2}
\mathcal F(s) =  F_0+ s\, F_1 + s^2\, F_2 + O(s)\ ,
\end{equation}
where $ F_0,  F_1, F_2$ are distributions on $S\times S\times S$.

\begin{lemma} \ 

$i)$ $Supp(F_0)\subset \mathcal O_4$

$ii)$ $Supp(F_1) \subset \overline{\mathcal O_3}$

$iii)$ on $\mathcal O_0$, $F_2$ coincides with $\mathcal K_{\boldsymbol \alpha,\, \mathcal O_0}$.
\end{lemma}

\begin{proof} $i)$ was proved in Proposition \ref{suppK}.

For $s\neq 0$ and $\varphi\in \mathcal C^\infty(S\times S\times S)$ with $Supp(\varphi) \subset \mathcal O_0$, 
\begin{equation}\label{FversusK}
\mathcal F(s) (\varphi)= \frac{(-1)^{k+l}}{k!\,l!}\frac{1}{\Gamma(-k+s)}\frac{1}{\Gamma(-l+s)}\ \mathcal K_{\alpha(s), \mathcal O_0}(\varphi)\ .
\end{equation}
Let $s\rightarrow 0$ and $iii)$ follows. Now let $\varphi\in \mathcal C^\infty(S\times S\times S)$ such that $Supp(\varphi)\cap (\mathcal O_3\cup\mathcal O_4)=\emptyset$. As $\alpha_1,\alpha_2\notin -(n-1)-2\mathbb N$,  the expression $\mathcal K_{\boldsymbol \alpha(s)} (\varphi)$ can be defined by analytic continuation in $s$, and \eqref{FversusK} is still valid. When $s\rightarrow 0$, the right hand side has a zero of order 2. Hence
$F_0(\varphi) = F_1(\varphi) = 0$, thus proving $ii)$.
\end{proof}

\begin{lemma}\label{invI+II}  For any $g\in G$, 
 \begin{align} 
 F_0\circ \boldsymbol \pi_{\boldsymbol \lambda}(g) -  F_0&=0\label{F0inv}\\ F_1\circ \boldsymbol \pi_{\boldsymbol \lambda}(g) - F_1 &= -A_g F_0\label{F1inv}\\
 F_2 \circ \boldsymbol \pi_{\boldsymbol \lambda}(g) - F_2 &= -A_g F_1+\frac{1}{2}\,A_g^2 F_0\label{F2inv}
\end{align}
where $A_g$ is the function on $S\times S\times S$ given by
\begin{equation}
A_g =   \ln \kappa(g^{-1},x)+\ln \kappa(g^{-1},y) \ .
\end{equation}
\end{lemma}
\begin{proof} Recall that $\mathcal F(s) \circ \boldsymbol  \pi _{\boldsymbol \lambda(s)} (g)= \mathcal F(s)$. Now
\begin{equation}\label{Taylorpi}
\boldsymbol \pi_{\boldsymbol \lambda(s)}(g) = \kappa(g^{-1},x)^s\kappa(g^{-1},y)^s \boldsymbol \pi_{\boldsymbol \lambda}(g) = \big(1+sA_g+s^2\frac{A_g^2}{2}+O(\vert s\vert ^3)\big) \,\boldsymbol \pi_{\boldsymbol \lambda}(g)
\end{equation}

Now  use \eqref{TaylorF2} and \eqref{Taylorpi} to get the conclusion. 
\end{proof}
In order to prove Proposition \ref{noextI+II}, a lemma is needed.

\begin{lemma}\label{strongsuppF1}
 \[Supp\,(F_1) = \overline{\mathcal O_3}\ .
 \]
 \end{lemma}
 \begin{proof} By Lemma \ref{invI+II}, the restriction of the distribution $F_1$ to $\mathcal O_4'$ is $\boldsymbol \lambda$-invariant and is supported in $\mathcal O_3$. From this follows that either $Supp(F_1)=\overline{\mathcal O_3}$ or $Supp(F_1)\subset \mathcal O_4$. Let us assume that $Supp(F_1) \subset \mathcal O_4$ and get a contradiction.  If this was the case, then the distribution $F_1$ would be of finite (global) transverse order  along the submanifold $\mathcal O_4$. The function $p_{a_1,a_2,a_3}$ vanishes on $\mathcal O_4$ to an arbitrary given order provided $a_1+a_2+a_3$ is large enough. Hence $F_1(p_{a_1,a_2,a_3}) = 0$ as $a_1+a_2+a_3$ is large enough. 
   
 So to get a contradiction , it is enough to prove that $F_1 (p_{a_1,a_2,a_3}) \neq 0$ for  some triples $(a_1,a_2,a_3)$ with $a_1+a_2+a_3$ arbitrary large. 
 
 \begin{lemma}\label{aalarge}
  Let $a_1,a_2,a_3$ be chosen so that $0\leq a_3\leq l$ and $a_1+a_2+a_3>k$. Then  \[\widetilde {\mathcal K}_{\alpha(s)} (p_{a_1,a_2,a_3}) = C_{a_1,a_2,a_3}\, s+O(s^2) \quad \text{ as } s\rightarrow 0\ ,
 \]
 where the constant $C_{a_1,a_2,a_3}\neq 0$ for $a_2, a_3$ large enough.
 \end{lemma} 
 \begin{proof} A careful look at \eqref{Kpalpha} shows that 

 \[\widetilde {\mathcal K}_{\alpha_1,\alpha_2,\alpha_3+s} (p_{a_1,a_2,a_3}) = \Phi(s) (-k+s)_{a_1+a_2+a_3} (-l+s)_{a_3}
 \]
 where $\Phi(s)\rightarrow \Phi(0)$ and $\Phi(0) \neq 0$ provided $a_1,a_2$ are large enough. If $0\leq a_3\leq l$ then $(-l)_{a_3}\neq 0 $ and if $a_1+a_2+a_3>k$, then $(-k+s)_{a_1+a_2+a_3}$ has a \emph{simple} $0$ for $s=0$. The statement follows.
 \end{proof}
 From \eqref{defF} and \eqref{TaylorF2} follows
 $F_0(p_{a_1,a_2,a_3}) = 0,  F_1(p_{a_1,a_2,a_3}) \neq 0
 $
 under the same constraints on $a_1,a_2,a_3$ as in Lemma \ref{aalarge}.
 \end{proof}

We are now ready for the proof of Proposition \ref{noextI+II}.  As $\alpha_1,\alpha_2\notin -(n-1)-2\mathbb N$, the distribution $\mathcal K_{\boldsymbol \alpha, \mathcal O_0}$ can be extended (by analytic continuation) to a $\boldsymbol \lambda$-invariant distribution $\mathcal K_{\boldsymbol \alpha, \mathcal O'_3}$ on $\mathcal O_3'$. Moreover  it is the unique $\boldsymbol \lambda$-invariant extension to $\mathcal O_3'$ of $\mathcal K_{\boldsymbol \alpha, \mathcal O_0}$, by the same argument as in the proof of Proposition \ref{noextI}.

Suppose there exists a $\boldsymbol \lambda$-invariant distribution $ T$ on $S\times S\times S$ which extends $\mathcal K_{\boldsymbol \alpha,\, \mathcal O_0}$ to $S\times S\times S$. From the last remark follows that $ T$ also extends $\mathcal K_{\boldsymbol \alpha, \mathcal O'_3}$. Form the distribution $ S=  T- F_2$. Then the distribution $ S$ satisfies

\begin{equation}\label{Sinv}
 S \circ \boldsymbol \pi_{\boldsymbol \lambda}(g) - S = g F_1-\frac{1}{2}\,A_g^2 F_0
\end{equation}
and 
\begin{equation}
Supp(S)\subset \overline{ \mathcal O_3}
\end{equation}
 Restrict \eqref{Sinv} to $\mathcal O'_3$ to get
 \begin{equation}\label{S'inv}
 S' - S' \circ \boldsymbol \pi_{\boldsymbol \lambda}(g) = A_g\, F_1'\ ,
 \end{equation}
 where ${\  '}$ means "restriction to $\mathcal O'_3$". From \eqref{F1inv}, $ F_1'$ is $\boldsymbol \lambda$-invariant on $\mathcal O_3'$. Arguing as in the previous situations, $S'$ is smoothly supported in $\mathcal O_3$, as well as $ F'_1$. We now reproduce the argument of Proposition \ref{noextI}, getting a contradiction as, by Lemma  \ref{strongsuppF1}, $F_1'\neq 0$.
\end{proof}

The proof of Theorem \ref{singsupp} in the case of a pole of type I+II (and not in $Z$) requires some more work.

\begin{lemma} Let $\boldsymbol \alpha$ be a pole of typeI+II, and assume that $\boldsymbol \alpha\notin Z$. Let $\boldsymbol \lambda$ be the associated spectral parameter. Let $T$be a $\boldsymbol \lambda$-invariant distribution. Then $Supp(T)\subset \mathcal O_4$.
 \end{lemma}

\begin{proof}
Let $\boldsymbol \alpha= (\alpha_1,\alpha_2,\alpha_3)$ satisfying
\[\alpha_1+\alpha_2+\alpha_3 = -2(n-1)-2k, \alpha_3 =-(n-1) -2l
\]
where $k,l\in \mathbb N$ and assume moreover that $\boldsymbol \alpha\notin Z$. Let $\boldsymbol \lambda$ be its associated spectral parameter, and let $T$ be a $\boldsymbol \lambda$-invariant distribution on $S\times S\times S$. By the same argument as in the case of a generic pole of type I, we get that $T$ is supported in $\overline {\mathcal O_3}=\mathcal O_4\cup \mathcal O_3$. Let $T'$ be its restriction to $\mathcal O_4'$. Then $F_1'$ and $T'$ are two $\boldsymbol \lambda$-invariant distributions on $\mathcal O_4'$ which are supported in $\mathcal O_3$. Hence by Lemma \ref{invsing1}, $T'=cF_1'$ for some constant $c$. Assume $c\neq 0$, and substituting $\frac{1}{c} T$ to $T$, we assume that $c=1$. Consider then the distribution $S=T-cF_1$. Then $Supp(S)\subset \mathcal O_4$, and $S$ satisfies
\[S\circ \boldsymbol \pi_{\boldsymbol \lambda}(g) -S = AF_0\ .
\]
As argued previously, this implies that $S$ is smoothly supported on $\mathcal O_4$. Let $p$ be its transverse order along $\mathcal O_4$. As $S$ is of transverse order $2k$ (see \cite{bc}), $p\geq 2k$. Now 
the transverse symbols at $(\mathbf 1,  \mathbf 1,\mathbf 1)$ satisfy
 \[\big(e^{t(2k-p)} -1\big) \sigma_p( S)\big((\mathbf 1,  \mathbf 1,\mathbf 1), \boldsymbol \xi\big)  = 2t\sigma_p( F_1)\big((\mathbf 1,  \mathbf 1,\mathbf 1), \boldsymbol \xi\big) \]
 for any $t\in\mathbb R$. As $ F_0\neq 0$ (and hence $\sigma_p( F_0)\big((\mathbf 1,  \mathbf 1,\mathbf 1) \boldsymbol \xi\big)\neq 0$), letting $t$ go to infinity yields a contradiction. Hence $T$ is supported in $\mathcal O_4$.
\end{proof}

 \section {Multiplicity one results}

 \subsection{Type I}
 \begin{theorem} 
Let $\boldsymbol \alpha$ be a generic pole of type $I$, and let $\boldsymbol \lambda$ be its associated spectral parameter. Let $T\neq 0$ be a $\boldsymbol \lambda$-invariant distribution. There exists a constant $C$ such that $T = C \widetilde {\mathcal K}^{\boldsymbol \lambda}$.
\end{theorem}

\begin{proof} From Theorem \ref{singsupp}, $T$ is supported in $\overline{\mathcal O_3} = \mathcal O_3\cup \mathcal O_4$. Let $\mathcal O'_4 = S\times S\times S \setminus \mathcal O_4$, and recall that $\mathcal O_3$ is a regular submanifold in $\mathcal O_4'$. Let $T'$ be the  restriction of $T$ to $\mathcal O_4'$. It is a $\boldsymbol \lambda$-invariant distribution supported in $\mathcal O_3$. On the other hand, 
consider the restriction of $\widetilde {\mathcal K}_{\boldsymbol \alpha}$ to $\mathcal O_4'$. Then it is a $\boldsymbol \lambda$-invariant distribution supported in $\mathcal O_3$, and it is not $0$, otherwise, $\widetilde {\mathcal K}_{\boldsymbol \alpha}$ would be supported in $\mathcal O_4$, a fact that would imply, by Lemma \ref{invsing2} that $\boldsymbol \alpha$ is a pole of type II, which is \emph{not} the case. Now Lemma \ref{invsing1} implies that there exists a constant $C$ such $T' = C \widetilde {\mathcal K}_{\boldsymbol \alpha, \mathcal O_4'}$.

Now $T-C \widetilde {\mathcal K}^{\boldsymbol \lambda}$ is a $\boldsymbol \lambda$-invariant distribution supported in $\mathcal O_4$. As $\boldsymbol \lambda$ is not a pole of type II, Lemma   \ref{invsing2} forces $T-C \widetilde {\mathcal K}^{\boldsymbol \lambda} = 0$ and the conclusion follows.

\end{proof}

\subsection{Type II and type I+II}

\begin{theorem}\label{mult1II}
Let $\boldsymbol \lambda$ be a pole of type II and assume that $\boldsymbol \lambda\notin Z$. Let $T\neq 0$ be a $\boldsymbol \lambda$-invariant distribution. There exists a constant $C$ such that $T=C\widetilde {\mathcal K}^{\boldsymbol \lambda}$.
\end{theorem}

The proof is more difficult, and we need first to recall
the relation between $\boldsymbol \lambda$-invariant distributions supported in $\mathcal O_4$ and \emph{covariant bi-differential operators}, for which we refer to \cite{bc}.
 
 \begin{lemma}\label{lemmabidiff}
  Let $\boldsymbol \lambda\in \mathbb C^3$. Let $T\neq 0$ be a distribution on $S\times S\times S$ which is $\boldsymbol \lambda$-invariant and such that $Supp(T)\subset \mathcal O_4$. There exists a bi-differential operator $B : \mathcal C^\infty(S\times S)\longrightarrow C^\infty(S)$ which intertwines $(\pi_{\lambda_1}\otimes \pi_{\lambda_2})$ and $\pi_{-\lambda_3}$ such that, for $f\in \mathcal C^\infty(S\times S), g\in \mathcal C^\infty(S)$ 
 \begin{equation}\label{bidiff}
 T(f\otimes g) = \int_S Bf(x)\, g(x) \,dx\ .
 \end{equation}
 Conversely, given a bi-differential operator $B : \mathcal C^\infty(S\times S)\longrightarrow C^\infty(S)$ which intertwines $(\pi_{\lambda_1}\otimes \pi_{\lambda_2})$ and $\pi_{-\lambda_3}$, then \eqref{bidiff} defines a $\boldsymbol \lambda$-invariant distribution. 
\end{lemma}

Given $\lambda, \mu\in \mathbb C$, denote by $\mathcal B\mathcal D_G(\lambda,\mu; k)$ the space of bi-differential operators from $\mathcal C^\infty(S\times S)$ into $\mathcal C^\infty(S)$ which are covariant w.r.t. $\pi_\lambda\otimes \pi_\mu$ and $\pi_{\lambda+\mu+\rho+2k}$. 

\begin{proposition} Let $\boldsymbol \lambda$ be a pole of type II (including type I+II), and assume $\boldsymbol \lambda \notin Z$. Let $k$ be the integer such that $\lambda_1+\lambda_2+\lambda_3 = -\rho-2k$. Then
\[\dim Tri(\boldsymbol \lambda) = \dim\big( \mathcal B\mathcal D_G(\lambda_1, \lambda_2; k)\big)\ .\]
\end{proposition}
\begin{proof} Let $T\in Tri(\boldsymbol \lambda)$. By Theorem \ref{singsupp}, in both cases $Supp(T)\subset \mathcal O_4$. Hence the result follows from Lemma \ref{lemmabidiff}.
\end{proof}

\begin{lemma}\label{charZk}
 Let $k$ be a given integer. Let $\mathcal H_k$ be the plane in $\mathbb C^3$ given by the equation
\[\lambda_1+\lambda_2 +\lambda_3 =-\rho-2k\ .
\]
Then $\boldsymbol \lambda\in Z$ if and only if (at least) one of the following properties $i)$ to $vi)$ is verified

$i)$ $\lambda_1+\lambda_2 \in -k+\mathbb N$

$ii)$ $\lambda_1 \in -\rho-k-\mathbb N$

$iii)$ $\lambda_2 \in -\rho-k-\mathbb N$

$iv)$ $\lambda_1 =-\rho-l,\  l\in \mathbb N$ and $\lambda_2 \in \{-k, -k+1,\dots, -k+l\}$.

$v)$ $\lambda_2 =-\rho-l,\  l\in \mathbb N$ and  $\lambda_1 \in \{-k, -k+1,\dots, -k+l\}$.

$vi)$ $\lambda_1 = -k+p, \lambda_2 = -k+q$, $p,q\in \mathbb N, p+q\leq k$.

\end{lemma}

The proof is elementary and we omit it. Let
 \[Z_k = \{(\lambda_1,\lambda_2)\in \mathbb C^2, (\lambda_1,\lambda_2, -\rho-2k -\lambda_1-\lambda_2)\in Z\}\ .\]
The elements of $Z_k$ have to satisfy (at least) one of the properties $i)$ through $vi)$.

To finish the proof of Theorem \ref{mult1II}, it remains to prove the following result, which will be done in the next subsection.
\begin{theorem}\label{Zkuniq0}
 Let $(\lambda_1,\lambda_2) \in \mathbb C^2, (\lambda_1,\lambda_2) \notin Z_k$, where $k\in \mathbb N$. Then 
\[\dim {\mathcal BD}_G(\lambda_1,\lambda_2;k) = 1\ .
\]
\end{theorem}

\subsection {Multiplicity one for bi-differential operators}

For the first time in this article, we will use the \emph{noncompact picture}. Let $F\simeq \mathbb R^{n-1}$ be the tangent space of $S$ at the point $ \bf 1$. The stereographic projection is a diffeomorphism from $S\setminus \{ -\mathbf 1\}$ onto $F$. The action of $G$ is transferred to a (rational) action of $G$ on $F$. The representations $\pi_{\lambda}$ can also be transferred, and in fact for the problem at hand, it is enough to consider the derivate of the representation, viewed as a representation of $\mathfrak g$. For $X\in \mathfrak g$, $d\pi_\lambda(X)$ acts by a first order differential operator with polynomial coefficients. A covariant bi-differential operator on $S\times S$ is transferred to a covariant bi-differential operator on $F\times F$ and vice versa. In our context, these operators have been studied in \cite{or}, to which we refer for more details. Choose coordinates on $F$ with respect to some orthogonal basis of $F$, and consider the bi-differential operators on $F\times F$ given by
\[\Delta_{x\,x} = \sum_{j=1}^{n-1} \frac{\partial ^2}{\partial x_j^2},\quad \Delta_{x\,y} = \sum_{j=1}^{n-1} \frac{\partial^2}{\partial x_j \partial y_j},\quad\Delta_{y\,y} = \sum_{j=1}^{n-1} \frac{\partial ^2}{\partial y_j^2}\ .
\]

Let $D$ be a bi-differential operator on $F\times F$ which is covariant w.r.t. $(\pi_{\lambda_1}\otimes \pi_{\lambda_2}, \pi_{\lambda_1+\lambda_2 +\rho+2k})$ for some $k\in \mathbb N$. As a first (and elementary) reduction of the problem, the operator $D$ has to be of the form

\begin{equation}\label{Deltabidiff}
D= \sum_{r+t\leq k} c_{r,t} \Delta^r_{x\,x} \Delta^s_{x\,y}\Delta^t_{y\,y}
\end{equation}
where $s=k-r-t$ and the $c_{r,t}$ are constants. 

The next propositions are two main results obtained in \cite{or}. Notice that in their notation $\dim F = n$, and their index $\lambda$ for the representation corresponds to $\frac{\lambda}{n-1}-\frac{1}{2}$ in our notation.

\begin{proposition} The bi-differential operator $D$ given by\eqref{Deltabidiff} is covariant w.r.t. $(\pi_{\lambda_1}\otimes \pi_{\lambda_2}, \pi_{\lambda_1+\lambda_2+\rho+2k})$ if and only if the $(c_{r,t})$ satisfy the  system $(\mathcal S)$ of homogeneous linear equations : 

\begin{equation*}
\begin{aligned}
4(r+1)(r+1+\lambda_1)c_{r+1,\,t}&\\ +2(k-r-t)(k-r+t-1+\rho+\lambda_2) c_{r,\,t} &\\ -(k-r-t+1)(k-r-t) c_{r,\,t-1} &= 0
\end{aligned}
\qquad E^{(1)}_{r,t}
\end{equation*}
\begin{equation*}
\begin{aligned}
4(t+1)(t+1+\lambda_2)c_{r,\,t+1}&\\ +2(k-r-t)(k+r-t-1+\rho+\lambda_1) c_{r,\,t}&\\ -(k-r-t+1)(k-r-t) c_{r-1,\,t}&=0
\end{aligned}
\qquad E^{(2)}_{r,t}
\end{equation*}
for $r+t\leq k$.
\end{proposition}

\begin{proposition} Let $\lambda_1, \lambda_2$ such that 
 \[ \lambda_1,\lambda_2 \notin \{ -\rho,-\rho-1,\dots, -\rho-(k-1)\} \cup \{ -1,-2,\dots, -k\}\ .\]
 Then the system $(\mathcal S)$ has, up to a constant, a unique non trivial solution.
 \end{proposition}
 
 Hence, the proof of Theorem  \ref{Zkuniq0} is reduced to the following  statement.
 
 \begin{proposition}\label{Zkuniq}
  Let $\lambda_1,\lambda_2\in \mathbb C$ such that
 $\lambda_1$ or $\lambda_2$ belongs to  $\{ -\rho,-\rho-1,\dots, -\rho-(k-1)\} \cup \{ -1,-2,\dots, -k\}$, but $\quad (\lambda_1,\lambda_2)\notin Z_k$.
 Then the system $(\mathcal S)$ has, up to a constant a unique solution.
 \end{proposition}
The proof of Proposition \ref{Zkuniq} is a case-by-case proof. As the arguments are of the same type in each case, we have chosen to  present the worst case with full details, giving only  a sketch of proof in the others.

 \begin{proposition} Let $\lambda_1= - p,\lambda_2= -q$, where $1\leq p,q\leq  k$, and assume that $(\lambda_1,\lambda_2)\notin Z_k$. Then the system $\mathcal (S)$ has a unique solution, up to a constant.
 \end{proposition}

\begin{proof} 
 
As  $vi)$ in Lemma \ref{charZk} is not satisfied,  $p+q< k$. This fact will reveal crucial for the non vanishing of some of the coefficients of the equations $E^{(1)}_{r,t}$ or $E^{(2)}_{r,t}$.
 
 \begin{lemma} Assume previous conditions are satisfied. Then
 \[c_{r,t} = 0 \quad \text {if}\quad  0\leq r\leq p-1\quad \text {or} \quad 0\leq t\leq q-1\ .
 \]
 \end{lemma}
 \begin{proof} \ 
  
 {\bf Step 1.}
 
 Equation $E^{(1)}_{p-1,0}$ reads
 \[2(k-p+1)(k-p+\rho-q) c_{p-1,0} + 0 = 0
 \]
 which implies $c_{p-1,0}=0$.
 
  {\bf Step 2.}

 Use equations $E^{(1)}_{r,0}$ successively for $r=p-2,\dots,0$ to compute $c_{r,0}$. Notice that the coefficient in the equation corresponding to the unknown $c_{r,0}$ is equal to
 \[2(k-r)(k-r-q+\rho-1)\neq 0\ .
 \]
 Hence $c_{r,0} = 0$ for $0\leq r\leq p-1$.
 
  {\bf Step 3.}

 For $1\leq t\leq q-1$, and $0\leq r \leq p-1$, use equation $E^{(2)}_{r,t-1}$ to compute $c_{r,t}$. Notice that the coefficient of the unknown in the equation is equal to $4t(t-q)\neq 0$. Let first $t=1$ to obtain $c_{r,1}=0$ for $0\leq r\leq p-1$, and then increment $t$ by $1$ up to $q-1$ to obtain  $c_{r,2} = 0$ for $0\leq r\leq p-1$, \dots, $c_{r,q-1} = 0$ for $0\leq r\leq p-1$.
  
 {\bf Step 4.}

For $r=p, p+1,\dots, k-q+1$, consider the equations $E^{(2)}_{r,q-1}$ which reduce to 
  \[2(k-r-q+1)(k+r-p-q+\rho) c_{r, q-1}+ (k-r-q)(k-r-q+1) c_{r-1,q-1} = 0\ ,
  \]   
and hence $c_{r,q-1}=0$ for $p\leq r\leq k-q+1$.
 
 {\bf Step 5.}
 
 For $p\leq r\leq k-q+1$ and $0\leq t\leq q-1$, use  equation $E^{(2)}_{r,t}$  to compute $c_{r,t}$ . Notice that  the coefficient of the unknown in the equation is equal to \[2(k-r-t) (k-t-p +r+\rho-1) \neq 0\ .\] Start with $t=q-2$ and let $r=p$ to conclude that $c_{p,q-2}=0$, then increment $r$ by $1$ up to $k-q+1$ to get $c_{r,q-2} = 0$. Then increment $t$ by $-1$, \dots, to conclude that $c_{r,t}=0$ if 
$p\leq r\leq k-q+1$ and $0\leq t\leq q-1$.

{\bf Step 6.}

For $r>k-q+1$ and $0\leq t\leq q-1$, use equation $E^{(1)}_{r-1,t}$ to compute $c_{r,t}$. Notice that the coefficient of the unknown in the equation is equal to $4r(r-p)$ and conclude that $c_{r,t}= 0$ for $r>k-q+1$ and $0\leq t\leq q-1$.

Summarizing what has been done, $c_{r,t} = 0$ for $t\leq q-1$ and arbitrary $r$. 
By exchanging the role of $r$ and $t$ (and of $p$ and $q$) follows 
$c_{r,t} = 0$ for $0\leq r\leq p-1$ and $t$ arbitrary.
 \end{proof}
 Now, choose $c_{p,q}$ as principal unknown and show that all the remaining unknowns (i.e. $c_{r,t}$ for $r\geq p, t\geq q$) can be expressed in terms of the principal unknown.
 
Let $r \geq p$. To determine $c_{r+1,q}$, use $E^{(1)}_{r,q}$ which, as $c_{r,q-1} = 0$ reads
 \[4(r+1)(r+1-p) c_{r+1,q}+ 2(k-r-q)(k-r-1+\rho) c_{r,q} = 0\ .
 \]
Starting with $r=p$ and incrementing $r$ by $1$ up to $k-q-1$ ,  the equations allow to compute all $c_{r,q}$ with $r\geq p+1$ in term of $c_{p,q}$.

Symmetrically, compute the unknowns $c_{p,t}, t\leq q+1$.

For $r> p$ and $t>q$, use equation $E^{(2)}_{r, t-1}$ to compute $c_{r,t}$. Notice that the coefficient of the unknown in the equation is equal to 
$4t(t-q)\neq 0$. Start with $r=p, t=q+1$ and increment $r$ by $1$. Then increment $t$ by $1$, and repeat the process over $r$ \dots, to effectively compute $c_{r,t}, r>p, t>q$. This achieves the proof of the proposition.
 \end{proof}
 
 \begin{proposition}\label{Zkuniq2}
  Let $\lambda_1 =-\rho-k_1, \lambda_2 = -\rho-k_2$, where $0\leq k_1,k_2\leq k-1$ and assume that $(\lambda_1,\lambda_2)\notin Z$. Then the system $(\mathcal S)$ has a unique non trivial solution, up to a constant.
 \end{proposition}
 \begin{proof}
 As property $ iv)$ of Lemma \ref{charZk} is not true, $\lambda_1 \notin \{ -k,-k+1,\dots, -k+k_2\}$ and similarly $\lambda_2 \notin \{ -k,-k+1,\dots, -k+k_1\}$. As in the previous case, this is crucial for the non vanishing of some of the coefficients of the equations. 
 
 First, for $r\geq k-k_2 \text{ or } t\geq k-k_1 ,\quad c_{r,t}=0$, which is proved with  arguments similar to those used in the previous case. Then choose $c_{0,0}$ as principal unknown, and compute the $c_{r,t}$ for $0\leq r\leq k-k_2-1, 0\leq t\leq k-k_1-1$ in terms of $c_{0,0}$.
\end{proof}

Now, assume $\lambda_1 = -\rho-k_1$, $\lambda_2 = -k_2$, with $0\leq k_1\leq k-1$ and $1\leq k_2\leq k$, and moreover $(\lambda_1,\lambda_2) \notin Z$. As $\lambda_1+\lambda_2+\lambda_3=-\rho-2k$, $\lambda_3 = -2k +k_1+k_2$. By Zk $iv)$, $-k+k_1 < -k_2$, and hence $\lambda_3 = -k-(k_1-k-l_2)\notin \{-k, -k+1,\dots,-1\}$. Because of the symmetry of our problem in $(\lambda_1,\lambda_2,\lambda_3)$, $(\lambda_1,\lambda_3)$ or $(\lambda_2,\lambda_3)$ do not belong to $Z_k$, and proving the uniqueness statement for $(\lambda_1, \lambda_2)$ or for $(\lambda_1,\lambda_3)$ or for $(\lambda_2,\lambda_3)$ is equivalent. So if $\lambda_3 \in \{ -\rho,-\rho-1,\dots,-\rho -(k-1)\}$, the uniqueness statement for $(\lambda_1,\lambda_3)$ is obtained by Proposition \ref{Zkuniq2}. Otherwise $\lambda_3\notin \{ -\rho,-\rho-1,\dots,-\rho -(k-1)\} \cup \{ -1,-2,\dots, -k_1\}$, the uniqueness for $(\lambda_1,\lambda_3)$ is part of the situations to be analyzed next.

Consider now the case where $\lambda_1 =-p$ with $1\leq p\leq k$, but $\lambda_2\notin \{ -\rho,-\rho-1,\dots, -\rho-(k-1)\} \cup \{ -1,-2,\dots, -k\}$. By the same argument used in first part $c_{r,t}=0$ if $r\leq p-1$. Then choose $c_{p,0}$ as principal unknown, and then compute the other $c_{r,t}$ for $r\leq p$ in terms of $c_{p,0}$ as done in the proof of Proposition \ref{Zkuniq}. 

In  the case where $\lambda_1 =-\rho-p$ for $0\leq p\leq k-1$ but $\lambda_2 \notin \{ -\rho,-\rho-1,\dots, -\rho-(k-1)\} \cup \{ -1,-2,\dots, -k\}$, first prove that $c_{0,t} = 0$ for $t\geq k-k_1$. Choose $c_{0,0}$ as principal unknown and then solve the system as done in Proposition \ref{Zkuniq2}. This achieves the proof Theorem \ref{Zkuniq0}.
\smallskip
\noindent

\subsection{Final remarks}

\begin{corollary}Let $\boldsymbol \lambda \in \mathbb C^3$ and assume that  the three representations $\pi_{\lambda_j}$ are irreducible. Then $\dim Tri(\boldsymbol \lambda) = 1$.
\end{corollary}
\begin{proof} The assumption of irreducibility amounts to
\[\lambda_j \notin (-\rho-\mathbb N)\cup(\rho+\mathbb N)
\]
for $j=1,2,3$. These conditions guarantee that $\boldsymbol \lambda$  is not in $Z$. Hence the result follows from Theorem \ref{mult1}.
\end{proof}
The next result  is  based on a remark due to T. Oshima (personal communication).

\begin{theorem}\label{mult2} Let $\boldsymbol \mu$ be in $Z$. Then
\[\dim Tri(\boldsymbol \mu)\geq 2\ .
\]
\end{theorem} 

\begin{proof} (Sketch) It is a consequence of Lemma 6.3 in \cite{o}. We  follow the notation of this paper. Let $U=\mathbb C^3$ and let
\[r(\boldsymbol \lambda)= \frac{1}{\Gamma(\frac{\lambda_1+\lambda_2+\lambda_3+\rho}{2})\Gamma(\frac{-\lambda_1+\lambda_2+\lambda_3+\rho}{2})\Gamma(\frac{\lambda_1-\lambda_2+\lambda_3+\rho}{2})\Gamma(\frac{\lambda_1+\lambda_2-\lambda_3+\rho}{2})}
\]
(this is nothing but the normalizing factor for $\mathcal K^{\boldsymbol \lambda}$). Let
$U_r = \{ \boldsymbol \lambda, r(\boldsymbol \lambda) \neq 0\}$. If ${\boldsymbol \lambda}\in U_r$, then $\lambda$ is not a pole and  hence $V_{\boldsymbol \lambda}= Tri({\boldsymbol \lambda})$ is of dimension 1. For $\boldsymbol \mu\in U\setminus U_r$, let $\overline V_{\boldsymbol \mu}$ be the \emph{closure} of the holomorphic family $V_{\boldsymbol \lambda}$ (see precise definition in \cite {o}). Then $\overline V_{\boldsymbol \mu} \subset Tri(\boldsymbol \mu)$. Now, by \cite{o},  $\dim \overline V_{\boldsymbol \mu} \geq \dim_{\lambda\in U_r} V_{\boldsymbol \lambda}=1$, and there is a necessary and sufficient condition for having equality (in which case $\boldsymbol \mu$ is said to be a \emph{removable} point). Now $Z$ is contained in $U\setminus U_r$. Let $\boldsymbol \mu$ be in $Z$. The fact that $\widetilde K^{\boldsymbol \mu}=0$ implies that $\boldsymbol \mu$ does \emph{not} satisfy the criterion for removability, so that $\dim \overline V_{\boldsymbol \mu} > \dim V_{\boldsymbol \lambda}$. All together, $\dim Tri(\boldsymbol \mu)>1$.
 \end{proof}
{\bf Final remark.} When $\boldsymbol \lambda$ is a pole of type II, and $\boldsymbol \lambda$ is in $Z$, a $\boldsymbol \lambda$-invariant distribution is not necessarily supported on the diagonal $\mathcal O_4$ (such cases will appear in \cite{c2}). So it is conceivable that for some $\boldsymbol \lambda$ in $Z$ 
$\dim \mathcal B\mathcal D_G(\lambda_1,\lambda_2; k)=1$  although $\dim Tri(\boldsymbol \lambda) >1$ by Theorem \ref{mult2}. So an explicit  necessary and sufficient condition for having   $\dim {\mathcal BD}(\lambda_1,\lambda_2;k)=1$ is still an open problem. This problem is to be studied in relation with a problem about tensor products of generalized Verma modules (see \cite{s} for some work in this direction).

\medskip
\footnotesize{\noindent Address\\ Jean-Louis Clerc\\Institut Elie Cartan, Universit\'e de Lorraine, 54506 Vand\oe uvre-l\`es-Nancy, France\\}
\medskip
\noindent \texttt{{jean-louis.clerc@univ-lorraine.fr
}}

\end{document}